\documentclass[11pt]{amsart}
\usepackage{amsmath,amssymb}
\usepackage{geometry}
\usepackage{enumitem}
\usepackage[hidelinks]{hyperref}
\usepackage{tikz}

\newtheorem{theorem}{Theorem}[section]
\newtheorem{lemma}[theorem]{Lemma}
\newtheorem{proposition}[theorem]{Proposition}
\newtheorem{corollary}[theorem]{Corollary}

\theoremstyle{definition}
\newtheorem{definition}[theorem]{Definition}
\newtheorem{remark}[theorem]{Remark}
\newtheorem{example}[theorem]{Example}

\newcommand{\K}{\mathbb K}
\newcommand{\cC}{\mathcal C}
\newcommand{\reg}{\operatorname{reg}}
\newcommand{\pd}{\operatorname{pd}}
\newcommand{\depth}{\operatorname{depth}}
\newcommand{\height}{\operatorname{ht}}
\newcommand{\bight}{\operatorname{bight}}
\newcommand{\Min}{\operatorname{Min}}
\newcommand{\Ass}{\operatorname{Ass}}
\newcommand{\N}{N}
\newcommand{\adeg}{\operatorname{adeg}}
\newcommand{\br}[2]{[\,#1,#2\,]}

\begin{document}

\title[Binomial edge ideals of bipartite complements of even cycles]
{Binomial Edge Ideals of Bipartite Complements of Even Cycles}

	\author[S.A. Rather]{Shahnawaz Ahmad Rather}
\address{Department of Mathematics \\
	Government Degree College  \\
	Baramulla \\
	Jammu and Kashmir 193101, India}
\email{nawaaz315@gmail.com}

\author{S. Pirzada}
\address{Department of Mathematics \\
	University of Kashmir, Hazratbal, \\
	Jammu and Kashmir 190006, India}
\email{pirzadasd@kashmiruniversity.net}
\author{M. Aijaz}
\address{Department of Mathematics \\
	Government Degree College for Women  \\
	Baramulla \\
	Jammu and Kashmir 193101, India}
\email{ahaijaz99@gmail.com}

\begin{abstract}
	Let $G_n$ be the bipartite complement of the even cycle $C_{2n}$,
	and let $J_{G_n}$ be its binomial edge ideal. For $n\geq5$, we study
	the interaction between the combinatorial structure of $G_n$ and the
	algebraic invariants of $S_{G_n}/J_{G_n}$. Our main combinatorial
	ingredient is a classification of the subsets of $V(G_n)$ having the
	cut point property, obtained through an analysis of disconnected
	induced subgraphs of $G_n$. We use this classification to describe the minimal primes and to study
	dimension and degree-theoretic invariants of $J_{G_n}$ and
	$S_{G_n}/J_{G_n}$. We also study local Vasconcelos numbers,
	graded Betti numbers and the Hilbert series, and obtain information
	on projective dimension and depth. Finally, we analyze induced paths
	in $G_n$ and derive corresponding bounds for the Castelnuovo--Mumford
	regularity.
\end{abstract}

\maketitle
\section{Introduction}\label{sec:introduction}

Let $G$ be a finite simple graph on the vertex set
$[m]=\{1,\ldots,m\}$, and let
\[
S_G=\K[x_1,\ldots,x_m,y_1,\ldots,y_m].
\]
The \emph{binomial edge ideal} of $G$ is
\[
J_G=(x_i y_j-x_j y_i:\{i,j\}\in E(G),\ i<j).
\]
Binomial edge ideals were introduced independently by Herzog, Hibi,
Hreinsd\'ottir, Kahle, and Rauh~\cite{HerzogEtAl} and by
Ohtani~\cite{Ohtani}. A central problem is to relate the combinatorial
structure of $G$ to the algebraic and homological properties of
$S_G/J_G$.

The minimal primes of $J_G$ are indexed by subsets of $V(G)$ having
the cut point property~\cite{HerzogEtAl}. Thus their study leads
naturally to understanding how vertex deletion affects the connected
components of $G$. A useful setting for this problem is provided by
structured families of graphs. For instance, Kumar, Pomeroy, and
Tran studied binomial edge ideals of crown graphs~\cite{Crown}, while
Hibi and Saeedi Madani investigated Cohen--Macaulayness, unmixedness,
dimension, and depth for Cameron--Walker graphs~\cite{CameronWalker}.
Betti numbers and regularity for several classes of graphs were
studied by Zafar and Zahid~\cite{Cycles}. Related work on
$v$-numbers includes that of Dey, Jayanthan, and Saha, who studied
this invariant for several classes of graphs, including paths and
cycles~\cite{VNumberClasses}.

In this paper we consider the bipartite complement of an even cycle.
Let
\[
X=\{u_0,\ldots,u_{n-1}\},
~
Y=\{v_0,\ldots,v_{n-1}\},
\]
and let $F_n\cong C_{2n}$ be the cycle with
\[
E(F_n)
=
\bigl\{
\{u_i,v_i\},\{u_i,v_{i-1}\}:i\in\mathbb Z_n
\bigr\},
\]
where $\mathbb Z_n=\mathbb Z/n\mathbb Z$, and all indices are read
modulo $n$.
We define
\[
G_n=K_{X,Y}\setminus E(F_n).
\]
Thus $G_n$ is obtained from $K_{n,n}$ by deleting two perfect
matchings whose union is a single even cycle. Throughout the paper
we assume $n\geq5$.

The cyclic arrangement of the missing edges makes the cut point
structure of $G_n$ substantially richer than that of the crown
graph. Our main combinatorial step is to analyze disconnected induced
subgraphs of $G_n$ through the deleted cycle $F_n$. Since every vertex
has degree two in $F_n$, such disconnections force a small number of
rigid configurations. This leads to a complete classification of the
cut point sets of $G_n$ into seven mutually exclusive types.

This classification is the organizing result of the paper. It yields
a description of all minimal primes of $J_{G_n}$ and allows us to
determine the dimensions and heights of the corresponding
minimal-prime quotients, the height and big height of $J_{G_n}$,
and the dimension, multiplicity, and arithmetic degree of
$S_{G_n}/J_{G_n}$. We also determine the
local Vasconcelos numbers, using separators and connected domination
in suitable induced subgraphs.

We further investigate the graded and homological structure of
$S_{G_n}/J_{G_n}$. The bipartite structure of $G_n$ restricts the
linear strand, while induced subgraphs on four vertices control
$\beta_{2,4}$. We also obtain a finite expression for the Hilbert
series and bounds for projective dimension and depth. Finally, we
study longest induced paths in $G_n$ and use them to obtain bounds
for the Castelnuovo--Mumford regularity.

The paper is organized as follows.
Section~\ref{sec:preliminaries} recalls the required preliminaries.
Section~\ref{sec:family} introduces $G_n$ and studies connectivity
of its induced subgraphs.
Section~\ref{sec:cutpoints} classifies the sets with the cut point
property and derives the resulting dimension, height, multiplicity,
and arithmetic-degree formulas.
Section~\ref{sec:graded} studies graded Betti numbers, the Hilbert
series, projective dimension, and depth.
Section~\ref{sec:vnumber} concerns connected domination and local
$v$-numbers.
Finally, Section~\ref{sec:regularity} studies induced paths and
regularity.

\section{Preliminaries}\label{sec:preliminaries}

Throughout the paper, all graphs are finite and simple, and $\K$ is
an arbitrary field. We recall the notation and basic facts used below.

\subsection{Graph-theoretic notation}

A graph $G$ is \emph{bipartite} if \( V(G)=X\sqcup Y \)
and every edge has one endpoint in $X$ and the other in $Y$. The pair
$(X,Y)$ is called a bipartition of $G$.

For $A\subseteq V(G)$, we denote by $G[A]$ the induced subgraph of
$G$ on $A$, and write \( G\setminus A=G[V(G)\setminus A]. \)
For a vertex $w$, we abbreviate $G\setminus\{w\}$ to $G\setminus w$.
The open neighbourhood and degree of $w$ are
\[
N_G(w)=\{z\in V(G):\{w,z\}\in E(G)\},
~
\deg_G(w)=|N_G(w)|.
\]
A vertex $w$ is a \emph{cut vertex} if $G\setminus w$ has more
connected components than $G$.

We denote by $K_r$, $P_r$, and $C_r$ the complete graph, path, and
cycle on $r$ vertices, respectively; thus $P_r$ has length $r-1$.
The disjoint union of graphs is denoted by $\sqcup$, so, for example,
$2P_3=P_3\sqcup P_3$.

For disjoint sets $X$ and $Y$, let $K_{X,Y}$ denote the complete
bipartite graph with bipartition $X\sqcup Y$. A \emph{matching} is a
set of pairwise vertex-disjoint edges, and it is \emph{perfect} if
every vertex is incident with exactly one edge of the matching.

If $F$ is a bipartite graph with specified bipartition $X\sqcup Y$,
its \emph{bipartite complement} is the graph on $X\sqcup Y$ with
edge set \( E(K_{X,Y})\setminus E(F). \)

A vertex is \emph{simplicial} if its neighbours form a clique; a
vertex that is not simplicial is called \emph{internal}.

\subsection{Binomial edge ideals and cut point sets}

For $T\subseteq V(G)$, put \( \overline T=V(G)\setminus T, \)
and let $c_G(T)$ denote the number of connected components of
$G[\overline T]$, with $c_G(V(G))=0$. Let
\( G_1,\ldots,G_{c_G(T)} \)
be these components, and let $\widetilde G_j$ denote the complete
graph on $V(G_j)$. Set
\[
P_T(G)
=
(x_i,y_i:i\in T)
+
\sum_{j=1}^{c_G(T)}J_{\widetilde G_j}.
\]
Each $P_T(G)$ is a prime ideal containing $J_G$, and
\cite[Theorem~3.2]{HerzogEtAl} gives
\[
J_G=\bigcap_{T\subseteq V(G)}P_T(G).
\]
In particular, $J_G$ is radical.

A subset $T\subseteq V(G)$ has the \emph{cut point property} if
$T=\varnothing$, or if every $i\in T$ is a cut vertex of
\( G[\overline T\cup\{i\}]. \)
Equivalently, a nonempty set $T$ has the cut point property if and
only if
\[
c_G(T\setminus\{i\})<c_G(T)
~\text{for every }i\in T.
\]
We write
\[
\cC(G)
=
\{T\subseteq V(G):T\text{ has the cut point property}\}.
\]

By \cite[Corollary~3.9]{HerzogEtAl}, applied to the connected
components when necessary, the minimal primes of $J_G$ are precisely
the ideals $P_T(G)$ with $T\in\cC(G)$. Hence
\[
J_G=\bigcap_{T\in\cC(G)}P_T(G)
\]
is an irredundant prime decomposition, and
\[
\Ass(S_G/J_G)
=
\Min(S_G/J_G)
=
\{P_T(G):T\in\cC(G)\}.
\]

Moreover, \cite[Lemma~3.1]{HerzogEtAl} gives
\begin{align}
	\height P_T(G)
	&=
	|V(G)|+|T|-c_G(T),
	\label{eq:height-prime}\\
	\dim(S_G/P_T(G))
	&=
	|V(G)|-|T|+c_G(T).
	\label{eq:dimension-prime}
\end{align}
Consequently,
\[
\dim(S_G/J_G)
=
\max_{T\in\cC(G)}
\bigl\{|V(G)|-|T|+c_G(T)\bigr\}.
\]

The following reformulation of the cut point property will be used
repeatedly.

\begin{lemma}\label{lem:component-criterion}
	Let $T\subseteq V(G)$ be nonempty, and let
	$H_1,\ldots,H_r$ be the connected components of $G[\overline T]$.
	Then $T\in\cC(G)$ if and only if every $v\in T$ has neighbours in at
	least two of $H_1,\ldots,H_r$.
\end{lemma}

\begin{proof}
	Fix $v\in T$, and suppose that $v$ has neighbours in exactly $s$ of
	the components of $G[\overline T]$. If $s\geq1$, restoring $v$
	merges these $s$ components into one, and therefore
	\[
	c_G(T\setminus\{v\})=c_G(T)-s+1.
	\]
	If $s=0$, restoring $v$ creates a new isolated component. Hence
	\( c_G(T)>c_G(T\setminus\{v\}) \)
	if and only if $s\geq2$. Applying this to every $v\in T$ proves the
	assertion.
\end{proof}

We will also use the column multigrading on $S_G$ defined by
\[
\deg x_i=\deg y_i=e_i\in\mathbb N^{V(G)}.
\]
For $\alpha\in\mathbb N^{V(G)}$, we write
\[
\operatorname{supp}(\alpha)
=
\{i\in V(G):\alpha_i>0\}.
\]

\subsection{Algebraic invariants}

Let $R=S_G/J_G$, and write its minimal graded free resolution as
\[
0\longrightarrow
\bigoplus_jS_G(-j)^{\beta_{p,j}(R)}
\longrightarrow\cdots\longrightarrow
\bigoplus_jS_G(-j)^{\beta_{0,j}(R)}
\longrightarrow R\longrightarrow0.
\]
The integers $\beta_{i,j}(R)$ are the graded Betti numbers. We use
\[
\pd R
=
\max\{i:\beta_{i,j}(R)\neq0\text{ for some }j\},
\qquad
\reg R
=
\max\{j-i:\beta_{i,j}(R)\neq0\}.
\]
Depth is taken at the homogeneous maximal ideal. Since
$\dim S_G=2|V(G)|$, the Auslander--Buchsbaum formula gives
\[
\depth R+\pd R=2|V(G)|.
\]
Unless stated otherwise, $\pd$, $\depth$, and $\reg$ refer to the
quotient $S_G/J_G$; when the ideal itself is considered, we write
$\reg(J_G)$ explicitly.

For an ideal $I\subseteq S_G$, its height and big height are
\[
\height(I)
=
\min\{\height(P):P\in\Min(I)\},
~
\bight(I)
=
\max\{\height(P):P\in\Min(I)\}.
\]
An ideal is \emph{unmixed} if all its associated primes have the same
height.

For a standard graded finitely generated $\K$-algebra
$R=\bigoplus_{d\geq0}R_d$, its Hilbert series is
\[
H_R(t)
=
\sum_{d\geq0}\dim_{\K}(R_d)t^d.
\]
If $d=\dim R$, then the Hilbert polynomial of $R$ has leading term
\[
\frac{e(R)}{(d-1)!}q^{d-1},
\]
where $e(R)$ denotes the multiplicity of $R$.

For the reduced ring $R=S_G/J_G$, its arithmetic degree is
\[
\adeg(R)
=
\sum_{P\in\Min(J_G)}e(S_G/P).
\]
\section{The graph \texorpdfstring{$G_n$}{G\_n} and its induced subgraphs}
\label{sec:family}

Fix an integer $n\geq2$.  Let \( X=\{u_0,u_1,\ldots,u_{n-1}\} \) and \( Y=\{v_0,v_1,\ldots,v_{n-1}\} \)
where all subscripts are taken modulo $n$.  Let $F_n$ be the cycle on
$X\sqcup Y$ whose edge set is \(  E(F_n)=\bigl\{\{u_i,v_i\},\{u_i,v_{i-1}\}:i\in\mathbb Z_n\bigr\}. \)
Thus $F_n\cong C_{2n}$.  We define \( G_n=K_{X,Y}\setminus E(F_n). \)
Equivalently,
\begin{equation}\label{eq:adjacency}
 \{u_i,v_j\}\in E(G_n)
 ~\text{if and only if}~
 j\notin\{i,i-1\}.
\end{equation}
In particular, $G_n$ is an $(n-2)$-regular bipartite graph on $2n$
vertices with \(  |E(G_n)|=n(n-2). \)

\begin{remark}
	The small cases are exceptional:
	$G_2$ is edgeless, $G_3\cong3K_2$, and $G_4\cong C_8$.
\end{remark}
Unless otherwise stated, we assume throughout the remainder
of this paper that $n\geq5$ and write $S=S_{G_n}$ and
$J=J_{G_n}$. All indices on $u_i$ and $v_i$ are read modulo $n$.

\subsection{Connectivity of induced subgraphs}\label{sec:connectivity}
For $U\subseteq V(G_n)$, write \( U_X=U\cap X,~ U_Y=U\cap Y, \)
and put $p=|U_X|$ and $q=|U_Y|$.  The graph $G_n[U]$ is the bipartite
complement of $F_n[U]$ inside $K_{p,q}$.

\begin{lemma}\label{lem:induced-connectivity}
Let $U\subsetneq V(G_n)$ with $|U_X|\geq3$ and $|U_Y|\geq3$.  If
$G_n[U]$ is disconnected, then \( |U_X|=|U_Y|=3 \) and \( F_n[U]\cong P_6 \) or 
\( F_n[U]\cong 2P_3. \)
In particular, if $|U_X|,|U_Y|\geq3$ and $|U|\geq7$, then $G_n[U]$
is connected.
\end{lemma}

\begin{proof}
Put $H=G_n[U]$, $p=|U_X|$, and $q=|U_Y|$, and suppose
that $H$ is disconnected.  Since every
vertex of $F_n$ has degree two, every vertex of $U_X$ is adjacent in
$H$ to at least $q-2\geq1$ vertices of $U_Y$.  Similarly, every vertex
of $U_Y$ is adjacent in $H$ to at least $p-2\geq1$ vertices of $U_X$.
Thus $H$ has no isolated vertices, and every component of $H$ meets
both $U_X$ and $U_Y$.

Let $H_1,\ldots,H_r$ be the connected components of $H$, and set
\[
 X_i=V(H_i)\cap U_X,~ Y_i=V(H_i)\cap U_Y.
\]
If $i\neq j$, then there are no edges of $G_n$ between $X_i$ and
$Y_j$.  Hence every pair in $X_i\times Y_j$ is an edge of $F_n$.
Since every vertex has degree two in $F_n$, it follows that \( |X_i|\leq2 \) and \( |Y_i|\leq2 \) for every \( i \). We first show that $r=2$.  Suppose $r\geq3$.  For $x\in X_1$, all
vertices belonging to $Y_2\cup\cdots\cup Y_r$ are neighbours of $x$ in
$F_n$.  Therefore \( |Y_2|+\cdots+|Y_r|\leq2. \)
Since each $Y_i$ is nonempty and $r\geq3$, we obtain $r=3$ and
$|Y_2|=|Y_3|=1$.  Applying the same argument with the other components
shows that $|Y_1|=1$.  By symmetry, $|X_1|=|X_2|=|X_3|=1$.  The edges
of $F_n$ between distinct components then form
$K_{3,3}$ minus a perfect matching, which is a $6$-cycle.  This is
impossible because $U\subsetneq V(F_n)$ and every proper induced
subgraph of the cycle $F_n$ is a disjoint union of paths.  Hence $r=2$.

Since $r=2$ and each $X_i$ and $Y_i$ has at most two
vertices, we have $3\leq p,q\leq4$. Suppose now that $p=4$.  Since $|X_1|,|X_2|\leq2$, we must have
$|X_1|=|X_2|=2$.  Because $q\geq3$, one of $Y_1,Y_2$ has two vertices.
The edges of $F_n$ between that two-element set and the two vertices
of the opposite $X$-component form a copy of $K_{2,2}$, contradicting
the fact that $F_n$ has no $4$-cycle.  Thus $p\neq4$, and the same
argument gives $q\neq4$.  Consequently, $p=q=3$.

The distributions of the two parts among the two components must be \( (|X_1|,|Y_1|)=(1,1) \), \( (|X_2|,|Y_2|)=(2,2) \), after possibly interchanging $H_1$ and $H_2$.  Indeed, if the
one-element $X$-part and the one-element $Y$-part occurred in different
components, the cross edges of $F_n$ would contain a $K_{2,2}$.

Write \( X_1=\{a\} \), \( Y_1=\{b\} \), \( X_2=\{c,d\}, \) and \( Y_2=\{e,f\}. \) 
The edges of $F_n$ between the two components of $H$ are
$ae,af,bc,bd$, and hence form the two paths $e-a-f$ and $c-b-d$.
The vertices $a$ and $b$ already have degree two in $F_n[U]$.
Consequently, every additional edge must join a vertex of
$\{c,d\}$ to a vertex of $\{e,f\}$.

Since $F_n$ has maximum degree two, any two additional edges
would have to be disjoint. Together with the four forced edges,
they would form a $6$-cycle, which is impossible in a proper
induced subgraph of $F_n$. Thus there is at most one additional
edge. With no additional edge we obtain $F_n[U]\cong 2P_3$,
and with one additional edge we obtain $F_n[U]\cong P_6$;
see Figure~\ref{fig:induced-connectivity}.
Finally, if $|U_X|,|U_Y|\geq3$ and $|U|\geq7$, then
$G_n[U]$ cannot be disconnected, since the preceding
argument would give $|U|=3+3=6$. In Figure \ref{fig:induced-connectivity}, the upper and lower rows
represent $U_X=\{a,c,d\}$ and $U_Y=\{b,e,f\}$, respectively.
The components of $G_n[U]$ have vertex sets $\{a,b\}$ and
$\{c,d,e,f\}$. The solid edges are forced edges of $F_n$ between
these components. In panel~(b), the additional dashed edge $ce$
joins the two paths into $f-a-e-c-b-d$.
All displayed edges belong to $F_n$, not to $G_n$.
\end{proof}

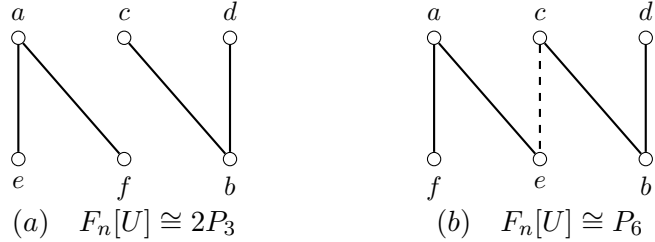
\begin{figure}[htb]
	\centering
	\begin{tikzpicture}[
		vertex/.style={
			circle,
			draw,
			fill=white,
			inner sep=0pt,
			minimum size=5pt
		},
		every label/.style={font=\small},
		forced/.style={thick},
		additional/.style={thick,dashed}
		]
		
		\begin{scope}
			\node[vertex,label=above:$a$] (a) at (0,1.6) {};
			\node[vertex,label=above:$c$] (c) at (1.4,1.6) {};
			\node[vertex,label=above:$d$] (d) at (2.8,1.6) {};
			
			\node[vertex,label=below:$e$] (e) at (0,0) {};
			\node[vertex,label=below:$f$] (f) at (1.4,0) {};
			\node[vertex,label=below:$b$] (b) at (2.8,0) {};
			
			\draw[forced] (a)--(e);
			\draw[forced] (a)--(f);
			\draw[forced] (b)--(c);
			\draw[forced] (b)--(d);
			
			\node at (1.4,-0.85) {$(a)\quad F_n[U]\cong 2P_3$};
		\end{scope}
		
		\begin{scope}[xshift=5.5cm]
			\node[vertex,label=above:$a$] (a) at (0,1.6) {};
			\node[vertex,label=above:$c$] (c) at (1.4,1.6) {};
			\node[vertex,label=above:$d$] (d) at (2.8,1.6) {};
			
			\node[vertex,label=below:$f$] (f) at (0,0) {};
			\node[vertex,label=below:$e$] (e) at (1.4,0) {};
			\node[vertex,label=below:$b$] (b) at (2.8,0) {};
			
			\draw[forced] (a)--(e);
			\draw[forced] (a)--(f);
			\draw[forced] (b)--(c);
			\draw[forced] (b)--(d);
			\draw[additional] (c)--(e);
			
			\node at (1.4,-0.85) {$(b)\quad F_n[U]\cong P_6$};
		\end{scope}
		
	\end{tikzpicture}
	\caption{The two possible configurations of $F_n[U]$ in
		Lemma~\ref{lem:induced-connectivity}. }
	\label{fig:induced-connectivity}
\end{figure}

We now use Lemma~\ref{lem:induced-connectivity} to
determine the vertex connectivity of $G_n$.

\begin{proposition}\label{prop:vertex-connectivity}
	The vertex connectivity of $G_n$ is $\kappa(G_n)=n-2$.
\end{proposition}

\begin{proof}
	Let $T\subseteq V(G_n)$ with $|T|\leq n-3$, and put
	$U=V(G_n)\setminus T$. Then \( |U_X|\geq3 \), \( |U_Y|\geq3 \) and \( |U|\geq n+3\geq8. \) If $T=\varnothing$, then $G_n[U]=G_n$ is connected:
	any two vertices of $X$ have at least $n-4\geq1$
	common neighbours in $Y$, and every vertex of $Y$
	has positive degree. If $T\neq\varnothing$, connectivity
	follows from Lemma~\ref{lem:induced-connectivity}.
	Thus removing fewer than $n-2$ vertices cannot
	disconnect $G_n$, and $\kappa(G_n)\geq n-2$.
	
	Conversely, $|\N_{G_n}(u_i)|=n-2$ for every
	$i\in\mathbb Z_n$. Deleting $\N_{G_n}(u_i)$ isolates
	$u_i$ while leaving other vertices in the graph.
	Hence $\kappa(G_n)\leq n-2$, proving the equality.
\end{proof}

\section{Cut point sets and algebraic consequences}
\label{sec:cutpoints}

In this section, we give a complete description of $\cC(G_n)$.
We first characterize the nonempty sets with the cut point
property that are contained in one bipartition class, and then
consider those meeting both classes. Combining these results
yields the classification in
Theorem~\ref{thm:cut-point-classification}.

Throughout this section, $F_n[U]$ denotes the subgraph of the
deleted cycle induced by $U$.

\begin{proposition}\label{prop:one-sided-cutsets}
	Let $\varnothing\neq T\subseteq Y$, and put
	$Y'=Y\setminus T$. Then $T\in\cC(G_n)$ if and only if
	one of the following holds:
	\begin{enumerate}[label=\textup{(\roman*)}]
		\item $Y'=\varnothing$;
		\item $Y'=\{v_i\}$ for some $i\in\mathbb Z_n$;
		\item $Y'=\{v_i,v_{i+1}\}$ for some $i\in\mathbb Z_n$.
	\end{enumerate}
	The analogous statement holds for
	$\varnothing\neq T\subseteq X$ after interchanging
	$X$ and $Y$.
\end{proposition}

\begin{proof}
	Put $U=V(G_n)\setminus T=X\cup Y'$ and $q=|Y'|$.
	Suppose first that $T\in\cC(G_n)$.
	Since $T\neq\varnothing$, the graph $G_n[U]$ is
	disconnected. Proposition~\ref{prop:vertex-connectivity}
	therefore gives $|T|\geq n-2$, and hence
	$q=n-|T|\leq2$.
	
	If $q=0$ or $q=1$, then condition \textup{(i)} or
	\textup{(ii)} holds. Suppose $q=2$, and write
	$Y'=\{v_i,v_j\}$.
	The vertices $v_i$ and $v_j$ have at least $n-4\geq1$
	common neighbours in $X$ with respect to $G_n$.
	Thus they and all their $G_n$-neighbours belong to
	one connected component. The remaining vertices
	are precisely the isolated vertices in
	$\N_{F_n}(v_i)\cap\N_{F_n}(v_j)$.
	Since $G_n[U]$ is disconnected, this intersection
	is nonempty. Using
	$\N_{F_n}(v_k)=\{u_k,u_{k+1}\}$, we obtain
	$j=i+1$ or $i=j+1$ modulo $n$.
	Consequently, condition \textup{(iii)} holds.
	
	Conversely, suppose that one of the three conditions
	holds. We verify the component criterion in
	Lemma~\ref{lem:component-criterion}.
	If $Y'=\varnothing$, then $G_n[U]$ consists of the
	$n$ isolated vertices of $X$.
	Every vertex of $T=Y$ has $n-2\geq3$ neighbours
	in distinct components. Hence $T\in\cC(G_n)$.  If $Y'=\{v_i\}$, then $G_n[U]$ consists of a star
	with centre $v_i$, together with the two isolated
	vertices $u_i$ and $u_{i+1}$.
	Let $v_j\in T$.
	The vertices $v_i$ and $v_j$ have a common
	$G_n$-neighbour in $X$, so $v_j$ has a neighbour
	in the star component. It is also adjacent in
	$G_n$ to at least one of $u_i,u_{i+1}$.
	Otherwise,
	\[
	\N_{F_n}(v_j)=\{u_i,u_{i+1}\}
	=\N_{F_n}(v_i),
	\]
	contradicting the distinctness of the
	$F_n$-neighbourhoods of distinct vertices of $Y$.
	Thus every vertex of $T$ meets at least two
	components, and $T\in\cC(G_n)$.
	
	Finally, suppose $Y'=\{v_i,v_{i+1}\}$.
	Their unique common $F_n$-neighbour is $u_{i+1}$.
	As observed above, $G_n[U]$ therefore has exactly
	two components: the isolated vertex $u_{i+1}$ and
	one component containing all remaining vertices.
	Every vertex of $T$ is adjacent in $G_n$ to
	$u_{i+1}$, since
	$\N_{F_n}(u_{i+1})=\{v_i,v_{i+1}\}$.
	It also has a neighbour in the other component:
	it has $n-2\geq3$ neighbours in $X$, and only
	one vertex of $X$ lies outside that component.
	Hence $T\in\cC(G_n)$. The assertion for $T\subseteq X$ follows by symmetry.
\end{proof}

We next characterize the sets with the cut point property that
meet both bipartition classes. Their complements have at most
six vertices, and their structure is determined by the deleted cycle.

\begin{proposition}\label{prop:mixed-cutsets}
	Let $n\geq5$, and let $T\subseteq V(G_n)$ satisfy
	$T\cap X\neq\varnothing$ and $T\cap Y\neq\varnothing$.
	Put \( U=V(G_n)\setminus T, \) \( U_X=U\cap X, \) and  \( U_Y=U\cap Y. \)
	Then $T\in\cC(G_n)$ if and only if one of the following holds:
	\begin{enumerate}[label=\textup{(\roman*)}]
		\item $(|U_X|,|U_Y|)\in\{(3,2),(2,3)\}$ and
		$F_n[U]\cong P_5$;
		
		\item $(|U_X|,|U_Y|)\in\{(4,2),(2,4)\}$ and
		$F_n[U]\cong2P_3$, where the two vertices in the
		smaller part are the centres of the two paths;
		
	\item $|U_X|=|U_Y|=3$ and
	$\bigl(F_n[U]\cong P_6 \text{ or } F_n[U]\cong2P_3\bigr)$.
	\end{enumerate}
\end{proposition}

\begin{proof}
	Put $H=G_n[U]$, $p=|U_X|$, and $q=|U_Y|$.
	
	\smallskip
	\noindent\emph{Necessity.}
	Suppose that $T\in\cC(G_n)$.
	Since $T\cap Y\neq\varnothing$, the component criterion
	in Lemma~\ref{lem:component-criterion} implies that at
	least two components of $H$ contain vertices of $X$.
	Similarly, at least two components contain vertices
	of $Y$. In particular, $H$ is disconnected and $p,q\geq2$. If $p,q\geq3$, then
	Lemma~\ref{lem:induced-connectivity} applies because
	$U\subsetneq V(G_n)$. It follows that $p=q=3$ and
	$F_n[U]\cong P_6$ or $2P_3$, giving
	condition~\textup{(iii)}.
	
	It remains to consider the case where at least one
	of $p,q$ equals two. By symmetry, assume that $q=2$
	and write $U_Y=\{v_i,v_j\}$.
	Every vertex of $T\cap X$ must have neighbours in at
	least two components of $H$. Since $v_i$ and $v_j$
	are the only surviving vertices of $Y$, they must
	lie in distinct components, and every vertex of
	$T\cap X$ must be adjacent in $G_n$ to both of them.
	Hence,
	\[
	\N_{F_n}(v_i)\cup\N_{F_n}(v_j)\subseteq U_X.
	\]
	Conversely, a vertex of $U_X$ outside this union would
	be adjacent in $G_n$ to both $v_i$ and $v_j$, placing
	them in the same component of $H$. Therefore
	\begin{equation}\label{eq:union-neighborhood}
		U_X=\N_{F_n}(v_i)\cup\N_{F_n}(v_j).
	\end{equation}
	
	If $j=i+1$ or $j=i-1$ modulo $n$, then the two
	$F_n$-neighbourhoods meet in exactly one vertex.
	Equation~\eqref{eq:union-neighborhood} gives $p=3$,
	and $F_n[U]\cong P_5$. Thus
	condition~\textup{(i)} holds. Otherwise, the two $F_n$-neighbourhoods are disjoint.
	Equation~\eqref{eq:union-neighborhood} then gives
	$p=4$, and $F_n[U]\cong2P_3$, with $v_i$ and $v_j$
	as the centres of the two paths. This gives
	condition~\textup{(ii)}. In particular, $p=q=2$ is impossible, since the union
	in~\eqref{eq:union-neighborhood} has at least three
	vertices. This completes the necessity argument.
	
	\smallskip
	\noindent\emph{Sufficiency.}
	We verify that every vertex of $T$ has neighbours
	in at least two components of $H$.
	
	\smallskip
	\noindent\textup{(i)}
	By symmetry, suppose that $(p,q)=(3,2)$.
	Label the path $F_n[U]$ as \( a-v_i-b-v_j-c, \)
	where $a,b,c\in X$.
	The connected components of $H$ have vertex sets
	$\{a,v_j\}$, $\{c,v_i\}$, and $\{b\}$.
	Thus $H$ consists of two disjoint edges and an isolated vertex. Both $v_i$ and $v_j$ have their two $F_n$-neighbours
	in $U$. Hence every vertex of $T\cap X$ is adjacent
	in $G_n$ to both $v_i$ and $v_j$, and therefore has
	neighbours in the two edge components.
	Now let $w\in T\cap Y$.
	Since $\N_{F_n}(b)=\{v_i,v_j\}$, the vertex $w$ is
	adjacent in $G_n$ to $b$.
	It is also adjacent in $G_n$ to at least one of
	$a,c$. Otherwise, $F_n$ would contain the six-cycle \( a-v_i-b-v_j-c-w-a, \)
	which is impossible because $F_n=C_{2n}$ and $2n>6$.
	Thus $w$ has neighbours in the isolated component
	and in at least one edge component.
	Lemma~\ref{lem:component-criterion} now gives
	$T\in\cC(G_n)$.
	
	\smallskip
	\noindent\textup{(ii)}
	By symmetry, suppose that $(p,q)=(4,2)$.
	Write $U_Y=\{v_i,v_j\}$ and put \( A=\N_{F_n}(v_i) \) and  \( B=\N_{F_n}(v_j). \)
	Since $v_i$ and $v_j$ are the centres of the two
	paths in $F_n[U]$, we have
	$|A|=|B|=2$, $A\cap B=\varnothing$, and $U_X=A\cup B$.
	The graph $H$ consists of two copies of $K_{1,2}$,
	with vertex sets $A\cup\{v_j\}$ and $B\cup\{v_i\}$. Every vertex of $T\cap X$ is adjacent in $G_n$ to
	both $v_i$ and $v_j$, because all their
	$F_n$-neighbours lie in $U_X$.
	Thus every such vertex has neighbours in both components of $H$. Let $w\in T\cap Y$.
	If $w$ had no $G_n$-neighbour in $A$, then
	$A\subseteq\N_{F_n}(w)$.
	Both sets have cardinality two, so
	\[
	\N_{F_n}(w)=A=\N_{F_n}(v_i).
	\]
	This is impossible: two distinct vertices with the
	same two $F_n$-neighbours would give a four-cycle
	in $F_n$. Hence $w$ has a $G_n$-neighbour in $A$.
	The same argument gives a $G_n$-neighbour in $B$.
	Therefore every vertex of $T$ has neighbours in
	both components, and
	Lemma~\ref{lem:component-criterion} yields
	$T\in\cC(G_n)$.
	
	\smallskip
	\noindent\textup{(iii)}
	Suppose that $p=q=3$ and
	$F_n[U]\cong P_6$ or $2P_3$.
	In the latter case, the two centres lie in opposite
	bipartition classes, since each class contains
	three vertices.
	In either configuration, we may label \( U_X=\{a,c,d\} \) and \( U_Y=\{b,e,f\} \)
	so that the edges of $F_n[U]$ are
	$ae,af,bc,bd$, together with at most one additional
	edge between $\{c,d\}$ and $\{e,f\}$.
	Consequently, $H$ has exactly two components:
	the edge on $\{a,b\}$ and a graph on $\{c,d,e,f\}$
	isomorphic to $K_{2,2}$ or $P_4$. Let $w\in T\cap X$.
	Since $\N_{F_n}(b)=\{c,d\}$, the vertex $w$ is
	adjacent in $G_n$ to $b$.
	It is also adjacent in $G_n$ to at least one of
	$e,f$. Otherwise, $F_n$ would contain the four-cycle
	$a-e-w-f-a$.
	Thus $w$ has neighbours in both components of $H$.
	
	Similarly, if $w\in T\cap Y$, then
	$\N_{F_n}(a)=\{e,f\}$ implies that $w$ is adjacent
	in $G_n$ to $a$.
	Moreover, $w$ is adjacent in $G_n$ to at least one
	of $c,d$, since otherwise $F_n$ would contain the
	four-cycle $b-c-w-d-b$.
	Hence every vertex of $T$ has neighbours in both
	components of $H$.
	By Lemma~\ref{lem:component-criterion},
	$T\in\cC(G_n)$.
\end{proof}

We can now combine the preceding propositions into a
complete description of the sets with the cut point property.

\begin{theorem}\label{thm:cut-point-classification}
	Let $T\subseteq V(G_n)$, and put
	$U=V(G_n)\setminus T$, $U_X=U\cap X$, and $U_Y=U\cap Y$.
	Then $T\in\cC(G_n)$ if and only if one of the following
	mutually exclusive conditions holds:
	\begin{enumerate}[label=\textup{(\roman*)}]
		\item $T=\varnothing$.
		
		\item $U=X$ or $U=Y$.
		
		\item $U=X\cup\{v_i\}$ or $U=Y\cup\{u_i\}$
		for some $i\in\mathbb Z_n$.
		
		\item $U=X\cup\{v_i,v_{i+1}\}$ or
		$U=Y\cup\{u_i,u_{i+1}\}$
		for some $i\in\mathbb Z_n$.
		
		\item $(|U_X|,|U_Y|)\in\{(3,2),(2,3)\}$ and
		$F_n[U]\cong P_5$.
		
		\item $(|U_X|,|U_Y|)\in\{(4,2),(2,4)\}$ and
		$F_n[U]\cong2P_3$, where the two vertices in
		the smaller part are the centres of the
		two copies of $P_3$.
		
		\item $|U_X|=|U_Y|=3$ and
		$\bigl(F_n[U]\cong P_6 \text{ or } F_n[U]\cong2P_3\bigr)$.
	\end{enumerate}
\end{theorem}

\begin{proof}
	The empty set belongs to $\cC(G_n)$ by definition.
	Suppose that $T\neq\varnothing$.
	If $T$ is contained in one bipartition class, then
	Proposition~\ref{prop:one-sided-cutsets} gives precisely
	conditions \textup{(ii)--(iv)}.
	If $T$ meets both bipartition classes, then
	Proposition~\ref{prop:mixed-cutsets} gives precisely
	conditions \textup{(v)--(vii)}.
	The conditions are mutually exclusive because, up to
	interchanging $X$ and $Y$, the corresponding pairs
	$(|U_X|,|U_Y|)$ are \( (n,n), (n,0), (n,1), (n,2),
	 (3,2), (4,2),\) and \( (3,3)
	\) respectively, and these pairs are distinct.
\end{proof}

\begin{corollary}\label{cor:number-cutsets}
For $n\geq5$, \( |\cC(G_n)|=2n^2+n+3\), and hence \( |\Min(J_{G_n})|=2n^2+n+3. \)
\end{corollary}

\begin{proof}
	We count the seven mutually exclusive types in
	Theorem~\ref{thm:cut-point-classification}. Since
	$T=V(G_n)\setminus U$, counting the surviving sets $U$
	is equivalent to counting the sets $T$.
	
	Types \textup{(i)} and \textup{(ii)} contribute one and
	two sets, respectively. Type \textup{(iii)} contributes
	$2n$ sets, since we choose the entire surviving
	bipartition class and one vertex of the other class.
	Type \textup{(iv)} also contributes $2n$ sets: each
	bipartition class has $n$ unordered pairs of cyclically
	consecutive vertices.
	
	For type \textup{(v)}, an induced $P_5$ in $F_n$ consists
	of five consecutive vertices of the cycle. There are
	$2n$ distinct such vertex sets, so this type contributes
	$2n$ sets.
	
	For type \textup{(vi)}, suppose first that the two centres
	lie in $Y$. Their neighbourhoods in $F_n$ must be disjoint,
	which occurs precisely when their indices are not
	consecutive modulo $n$. Of the $\binom{n}{2}$ unordered
	pairs in $Y$, exactly $n$ have consecutive indices.
	The two centres uniquely determine the surviving set.
	Allowing the centres to lie in either bipartition class,
	we therefore obtain \( 2\left(\binom{n}{2}-n\right)=n(n-3) \)
	sets of this type.
	
	For type \textup{(vii)}, the induced copies of $P_6$
	contribute $2n$ sets, one for each set of six consecutive
	vertices of $F_n$.
	It remains to count the induced copies of $2P_3$ having
	three vertices in each bipartition class. In such a copy,
	one centre lies in $X$ and the other lies in $Y$.
	Fix the centre $u_i\in X$. Its path is
	$v_{i-1}-u_i-v_i$. If the other centre is $v_j\in Y$,
	its path is $u_j-v_j-u_{j+1}$.
	The two paths overlap when $j=i-1$ or $j=i$.
	When $j=i-2$, the edge $u_{i-1}v_{i-1}$ joins them;
	when $j=i+1$, the edge $u_{i+1}v_i$ joins them.
	For every other value of $j$, the paths are disjoint
	and there is no edge between them. Hence precisely
	$n-4$ choices of $v_j$ give an induced $2P_3$.
	Each such induced subgraph has a unique centre in $X$
	and a unique centre in $Y$, so it is counted exactly
	once. There are therefore $n(n-4)$ balanced copies of
	$2P_3$. Altogether, type \textup{(vii)} contributes
	$2n+n(n-4)=n(n-2)$ sets.
	Summing the seven contributions yields
	\[
	|\cC(G_n)|
	=1+2+2n+2n+2n+n(n-3)+n(n-2)
	=2n^2+n+3.
	\]
	The equality for $|\Min(J_{G_n})|$ follows from the
	bijection $T\mapsto P_T(G_n)$ between $\cC(G_n)$ and
	the minimal primes of $J_{G_n}$.
\end{proof}

\subsection{Dimension and height}\label{sec:dimension}

We now use the classification of sets with the cut point
property to determine the Krull dimension of
$S_{G_n}/J_{G_n}$ and the dimensions of its minimal-prime
quotients.

\begin{theorem}\label{thm:prime-dimensions}
	The Krull dimension of $S_{G_n}/J_{G_n}$ is $2n+1$.
	Moreover, for $T\in\cC(G_n)$, the following assertions hold:
	\begin{enumerate}[label=\textup{(\roman*)}]
		\item If $T=\varnothing$, then
		$\dim(S_{G_n}/P_T(G_n))=2n+1$.
		
		\item If $T=X$ or $T=Y$, then
		$\dim(S_{G_n}/P_T(G_n))=2n$.
		
		\item If $\varnothing\neq T\subsetneq X$ or
		$\varnothing\neq T\subsetneq Y$, then
		$\dim(S_{G_n}/P_T(G_n))=n+4$.
		
		\item If $T$ meets both bipartition classes, then
		$\dim(S_{G_n}/P_T(G_n))=8$.
	\end{enumerate}
\end{theorem}

\begin{proof}
	Put $U=V(G_n)\setminus T$.
	By \eqref{eq:dimension-prime}, we have
	\[
	\dim(S_{G_n}/P_T(G_n))
	=2n-|T|+c_{G_n}(T)
	=|U|+c_{G_n}(T).
	\]
	We compute $|U|$ and $c_{G_n}(T)$ in each case.
	
	\smallskip
	\noindent\textup{(i)}
	If $T=\varnothing$, then $U=V(G_n)$.
	Since $G_n$ is connected, $c_{G_n}(T)=1$, and hence
	$\dim(S_{G_n}/P_T(G_n))=2n+1$.
	
	\smallskip
	\noindent\textup{(ii)}
	Suppose that $T=X$; the case $T=Y$ is symmetric.
	Then $U=Y$, and $G_n[U]$ consists of $n$ isolated
	vertices. Thus $|U|=c_{G_n}(T)=n$, giving
	$\dim(S_{G_n}/P_T(G_n))=2n$.
	
	\smallskip
	\noindent\textup{(iii)}
	By symmetry, suppose that
	$\varnothing\neq T\subsetneq Y$.
	Proposition~\ref{prop:one-sided-cutsets} shows that
	either $U=X\cup\{v_i\}$ or
	$U=X\cup\{v_i,v_{i+1}\}$ for some $i\in\mathbb Z_n$.
	
	In the first case, $G_n[U]$ consists of a star on
	$n-1$ vertices and two isolated vertices.
	Therefore $|U|=n+1$ and $c_{G_n}(T)=3$, so
    \( \dim(S_{G_n}/P_T(G_n))=(n+1)+3=n+4. \)
	In the second case, $G_n[U]$ consists of a connected
	component on $n+1$ vertices and one isolated vertex.
	Hence $|U|=n+2$ and $c_{G_n}(T)=2$, giving
	\( \dim(S_{G_n}/P_T(G_n))=(n+2)+2=n+4. \)
	
	\smallskip
	\noindent\textup{(iv)}
	Suppose that $T$ meets both bipartition classes.
	By Proposition~\ref{prop:mixed-cutsets}, there are
	three possible configurations. If $F_n[U]\cong P_5$, then $G_n[U]$ consists of two disjoint edges and one isolated vertex.
	Thus $|U|=5$ and $c_{G_n}(T)=3$, yielding
	$\dim(S_{G_n}/P_T(G_n))=8$. In each of the other two configurations, $|U|=6$
	and $G_n[U]$ has exactly two components.
	More precisely, when the surviving part sizes are
	$(4,2)$ or $(2,4)$, these components are two copies
	of $K_{1,2}$. When the surviving part sizes are
	$(3,3)$, they are a copy of $K_2$ together with
	a copy of $P_4$ or $K_{2,2}$.
	Consequently,
	$\dim(S_{G_n}/P_T(G_n))=8$.
	
	Finally, Theorem~\ref{thm:cut-point-classification}
	accounts for all minimal primes of $J_{G_n}$.
	Therefore
	\[
	\begin{aligned}
		\dim(S_{G_n}/J_{G_n}) =\max_{T\in\cC(G_n)}
		\dim(S_{G_n}/P_T(G_n))=\max\{2n+1,\,2n,\,n+4,\,8\}=2n+1.
	\end{aligned}
	\]
\end{proof}
The dimension formula also admits a direct proof using
a Hamiltonian cycle.

\begin{proposition}\label{prop:Hamiltonian}
	The graph $G_n$ is Hamiltonian for every $n\geq4$.
\end{proposition}

\begin{proof}
	Consider the two perfect matchings \( M_1=\bigl\{\{u_i,v_{i+1}\}:i\in\mathbb Z_n\bigr\} \) and  \( M_2=\bigl\{\{u_i,v_{i+2}\}:i\in\mathbb Z_n\bigr\}. \)
	By \eqref{eq:adjacency}, all these edges belong to
	$G_n$. Moreover, $M_1$ and $M_2$ are disjoint, so
	their union is a spanning subgraph in which every
	vertex has degree two.
	
	Starting at $u_i$ and alternately following an edge of $M_1$
	and an edge of $M_2$, we obtain
	\[
	u_i \xrightarrow{M_1} v_{i+1}
	\xrightarrow{M_2} u_{i-1}.
	\]
	Thus the indices of the $X$-vertices decrease by one
	modulo $n$ after every two steps.
	Consequently, this traversal visits all $n$ vertices
	of $X$ and all $n$ vertices of $Y$ exactly once before
	returning to $u_i$.
	Hence $M_1\cup M_2$ is a Hamiltonian cycle of $G_n$.
\end{proof}

\begin{remark}
	Let $C$ be a Hamiltonian cycle of $G_n$, and let
	$\varnothing\neq T\subsetneq V(G_n)$.
	The graph $C\setminus T$ is a disjoint union of paths,
	where isolated vertices are allowed.
	Fix an orientation of $C$.
	The first vertex of each surviving path is immediately
	preceded along $C$ by a vertex of $T$, and distinct
	paths have distinct preceding vertices.
	Thus $C\setminus T$ has at most $|T|$ components. Since $C\setminus T$ is a spanning subgraph of
	$G_n\setminus T$, we obtain \( c_{G_n}(T)\leq |T|. \)
	Therefore, for every nonempty $T\in\cC(G_n)$,
	\[
	\dim(S_{G_n}/P_T(G_n))
	=2n-|T|+c_{G_n}(T)
	\leq2n.
	\]
	On the other hand, $G_n$ is connected, so
	$\dim(S_{G_n}/P_{\varnothing}(G_n))=2n+1$.
	Taking the maximum over the minimal primes therefore
	gives
	\[
	\dim(S_{G_n}/J_{G_n})=2n+1,
	\]
	without using the classification in
	Theorem~\ref{thm:cut-point-classification}.
\end{remark}

\begin{corollary}\label{cor:height-bight}
	The ideal $J_{G_n}$ satisfies
	$\height(J_{G_n})=2n-1$ and $\bight(J_{G_n})=4n-8$.
\end{corollary}

\begin{proof}
	Since $\dim S_{G_n}=4n$, every minimal prime
	$P_T(G_n)$ satisfies
	\[
	\height(P_T(G_n))
	=4n-\dim(S_{G_n}/P_T(G_n)).
	\]
Theorem~\ref{thm:prime-dimensions} therefore gives the possible
heights
\[
2n-1,~ 2n,~ 3n-4,~ 4n-8.
\]
Since $n\geq5$, these satisfy
\[
2n-1<2n<3n-4<4n-8.
\]
	The height and big height of $J_{G_n}$ are,
	respectively, the minimum and maximum heights
	among its minimal primes. Hence
	$\height(J_{G_n})=2n-1$ and
	$\bight(J_{G_n})=4n-8$.
\end{proof}

\begin{corollary}\label{cor:not-unmixed}
	The ideal $J_{G_n}$ is not unmixed.
	Consequently, $S_{G_n}/J_{G_n}$ is not Cohen--Macaulay.
\end{corollary}

\begin{proof}
	By Corollary~\ref{cor:height-bight}, the minimal
	primes of $J_{G_n}$ have different heights.
	Thus $J_{G_n}$ is not unmixed.
	Since a homogeneous ideal with a Cohen--Macaulay
	quotient is unmixed, $S_{G_n}/J_{G_n}$ cannot be
	Cohen--Macaulay.
\end{proof}

\subsection{Multiplicity and arithmetic degree}

The classification of sets with the cut point property
also allows us to compute the multiplicity and arithmetic
degree of $S_{G_n}/J_{G_n}$.

\begin{proposition}\label{prop:degrees}
	The multiplicity and arithmetic degree of
	$S_{G_n}/J_{G_n}$ are given by \( e(S_{G_n}/J_{G_n})=2n, \) and \( \adeg(S_{G_n}/J_{G_n})=21n^2-33n+2. \)
\end{proposition}

\begin{proof}
	Write $S=S_{G_n}$ and $J=J_{G_n}$.
	For $T\in\cC(G_n)$, let $r_1,\ldots,r_c$ be the
	orders of the connected components of $G_n\setminus T$.
	The quotient $S/P_T(G_n)$ is a tensor product over
	$\K$ of the complete-graph quotients corresponding
	to these components.
	Since $e(S_{K_r}/J_{K_r})=r$ and multiplicities
	multiply under tensor products over $\K$, we obtain
	\[
	e(S/P_T(G_n))=\prod_{j=1}^{c}r_j.
	\]
	The ring $S/J$ is reduced, so its associated primes
	are precisely its minimal primes, and its localization
	at each minimal prime is a field.
	Consequently, the local multiplicity factors appearing
	in the associativity formula are all equal to one.
	By Theorem~\ref{thm:prime-dimensions},
	$P_{\varnothing}(G_n)$ is the unique minimal prime
	whose quotient has dimension $\dim(S/J)$.
	Since $G_n$ is connected, the associativity formula
	therefore gives
	\[
	e(S/J)=e(S/P_{\varnothing}(G_n))=2n.
	\]
	
	For the arithmetic degree, every minimal prime
	contributes. Thus
	\[
	\adeg(S/J)
	=\sum_{T\in\cC(G_n)}e(S/P_T(G_n)).
	\]
	We evaluate this sum using the types in
	Theorem~\ref{thm:cut-point-classification} and
	the counts established in the proof of
	Corollary~\ref{cor:number-cutsets}. The empty set contributes $2n$.
	For each of the two sets of type \textup{(ii)},
	all surviving vertices are isolated, so the
	corresponding multiplicity is $1$.
	For types \textup{(iii)} and \textup{(iv)}, the
	component orders are $n-1,1,1$ and $n+1,1$,
	respectively. Each type contains $2n$ sets, so
	their contributions are $2n(n-1)$ and $2n(n+1)$.	The $2n$ sets of type \textup{(v)} have component
	orders $2,2,1$, giving multiplicity $4$ each.
	The $n(n-3)$ sets of type \textup{(vi)} have
	component orders $3,3$, giving multiplicity $9$ each.
	Finally, the $n(n-2)$ sets of type \textup{(vii)}
	have component orders $2,4$, giving multiplicity
	$8$ each. Therefore
	\[
	\begin{aligned}
		\adeg(S/J)
		=2n+2+2n(n-1)+2n(n+1)+8n+9n(n-3)+8n(n-2)=21n^2-33n+2.
	\end{aligned}
	\]
\end{proof}

\section{Graded and homological invariants}\label{sec:graded}
We now study the graded and homological structure of
$S_{G_n}/J_{G_n}$. We first determine the linear strand and
$\beta_{2,4}$ and obtain a finite expression for the Hilbert series.
We then use the big-height computation together with known
small-depth classifications to obtain bounds for projective dimension
and depth.
\subsection{Betti numbers and the Hilbert series}\label{sec:betti}
\begin{proposition}\label{prop:linear}
	The graded Betti numbers in the linear strand of \( S_{G_n}/J_{G_n} \) are given by
	\[
	\beta_{i,i+1}(S_{G_n}/J_{G_n})=
	\begin{cases}
		n(n-2), & i=1,\\
		0,      & i\geq2.
	\end{cases}
	\]
\end{proposition}

\begin{proof}
	The edge binomials form a minimal quadratic generating
	set of $J$. Since $|E(G_n)|=n(n-2)$, we obtain
	$\beta_{1,2}(S/J)=n(n-2)$. Equip $S$ with the column grading
	$\deg x_w=\deg y_w=e_w$.
	A relation with linear coefficients decomposes into
	multihomogeneous relations of total degree three,
	each supported on at most three vertices.
	Since $G_n$ is bipartite, at most two edge generators
	occur in such a component.
	A relation involving only one generator has the form $af=0$ with
	$f\neq0$, and hence $a=0$ since $S$ is a domain.
	If two generators $f,g$ occur, they are relatively prime irreducible
	quadrics. Thus $af+bg=0$ implies $f\mid b$ and $g\mid a$.
	Since $a$ and $b$ are linear, both must vanish. Hence $\beta_{2,3}(S/J)=0$.
	
	Let $\mathcal F_\bullet$ be the minimal graded free resolution
	of $S/J$. Its differentials have entries in the homogeneous maximal
	ideal $\mathfrak m$, and all shifts in $\mathcal F_i$ are at least
	$i+1$ for $i\geq1$.
	The preceding vanishing shows that the shifts in $\mathcal F_2$
	are at least four. Inductively, suppose that all shifts in $\mathcal F_{i-1}$
	are at least $i+1$.
	A basis element $z\in \mathcal F_i$ of degree $i+1$ would satisfy
	$\partial_i(z)=0$ by homogeneity and minimality.
	Exactness would then give
	\[
	z\in\operatorname{im}\partial_{i+1}
	\subseteq\mathfrak m \mathcal F_i,
	\]
	which is impossible for a basis element.
	Therefore $\beta_{i,i+1}(S/J)=0$ for every $i\geq2$.
\end{proof}

\begin{theorem}\label{thm:second-betti}
Set $m=n(n-2)$ and let \( q_n \) denote the number of induced four-cycles in \( G_n \). Then
\[
 q_n=n\binom{n-3}{2}+\frac{n(n-3)}2\binom{n-4}{2}.
\]
Moreover, \(  \beta_{2,4}(S_{G_n}/J_{G_n})
=\binom{m}{2}+2n\binom{n-2}{3}+3q_n. \)
\end{theorem}
\begin{proof}
	Since $G_n$ is triangle-free and has $m$ edges,
	Proposition~\ref{prop:linear} gives $\beta_{2,3}(S/J)=0$.
	Hence there are no second syzygies in total degree $3$, and the
	degree-$4$ part of the second syzygy module contributes directly to
	$\beta_{2,4}(S/J)$.
	
	In multidegree $2e_a+e_b+e_c$, the only contribution is the Koszul
	syzygy for incident edges $ab,ac$, when both are present. Other
	multidegrees with repeated vertices contribute nothing.
	
	On four distinct vertices $W$, each edge contributes four basis
	monomials to the degree-$\sum_{w\in W}e_w$ part of the free module
	of generators. If $H=G_n[W]$ has $r$ edges and component orders
	$s_1,\ldots,s_c$, the corresponding part of $S/J_H$ has dimension
$\prod_j(s_j+1)$. Indeed, the edge relations identify the column
monomials precisely according to the number of $x$'s in each
component. Moving an $x$ along an edge connects all configurations
with a fixed such number. The syzygy dimension is therefore
$4r-16+\prod_j(s_j+1)$.
The triangle-free graphs on four vertices are forests or $C_4$.
Among the triangle-free graphs on four vertices, the
squarefree syzygy dimension is one for $2K_2$, $P_4$,
and $K_{1,3}$, and five for $C_4$; it is zero for the
remaining graphs.
The graphs $2K_2$ and $P_4$ each contain one pair of
disjoint edges, while $C_4$ contains two.
Thus these contributions amount to one for each pair
of disjoint edges, one additional contribution for each
induced claw, and three additional contributions for each
induced four-cycle.
Including the incident-edge Koszul syzygies gives
\[
 \beta_{2,4}=\binom m2+\#\{\text{induced }K_{1,3}\}
                      +3\#\{\text{induced }C_4\}.
\]
In $G_n$, every induced claw is centered at a vertex, and each
vertex has degree $n-2$. Hence the number of induced claws is \( 2n\binom{n-2}{3}. \)
It remains to count the induced $4$-cycles. Fix an unordered pair
$\{u_i,u_j\}\subseteq X$. If $u_i$ and $u_j$ are cyclically
consecutive in $X$, then they have $n-3$ common neighbours in $Y$.
There are $n$ such pairs. If they are not cyclically consecutive,
then they have $n-4$ common neighbours, and there are \( \binom{n}{2}-n=\frac{n(n-3)}2 \)
such pairs. Choosing two common neighbours of a pair
$\{u_i,u_j\}$ produces an induced $4$-cycle, and every induced
$4$-cycle is obtained uniquely in this way. Therefore
\[
q_n
=
n\binom{n-3}{2}
+
\frac{n(n-3)}2\binom{n-4}{2}.
\]
Substituting the claw and $4$-cycle counts into the preceding formula
gives
\[
\beta_{2,4}(S/J)
=
\binom{m}{2}
+
2n\binom{n-2}{3}
+
3q_n.
\]
\end{proof}

\begin{proposition}\label{prop:Hilbert}
 For each $U\subseteq V(G_n)$, let $\pi(U)$
	denote the collection of vertex sets of the connected
	components of $G_n[U]$. Then the Hilbert series of
	$S_{G_n}/J_{G_n}$ is

	\[
	H_{S_{G_n}/J_{G_n}}(t)
	=\sum_{U\subseteq V(G_n)}
	\frac{t^{|U|}\prod_{C\in\pi(U)}(|C|+1-t)}
	{(1-t)^{|U|+|\pi(U)|}}.
	\]
Here $\pi(\varnothing)=\varnothing$, and the empty product
is interpreted as $1$, so the empty-set summand is $1$.
\end{proposition}

\begin{proof}
	Write $S=S_{G_n}$ and $J=J_{G_n}$, and equip $S$ with
	the multigrading $\deg x_w=\deg y_w=e_w$.
	Fix a multidegree
	$\alpha\in\mathbb N^{V(G_n)}$ with support
	$U=\{w:\alpha_w>0\}$.
	Every monomial of multidegree $\alpha$ has the form
	\[
	\prod_{w\in U}x_w^{k_w}y_w^{\alpha_w-k_w},
	\]
	where \( 0\leq k_w\leq\alpha_w \). For an edge $\{u,v\}$ of $G_n[U]$, the relation
	$x_u y_v-x_v y_u\in J$ identifies two such monomials
	whenever their exponent vectors differ by moving one
	unit from $k_u$ to $k_v$, or conversely, within the
	prescribed bounds.
	Within each connected component $C$, these moves
	preserve the sum $\sum_{w\in C}k_w$.
	
	Conversely, because each component is connected, the edge moves
	generate all transfers of units between vertices of that component.
	Using a spanning tree and moving units successively from leaves to the
	root shows that any two bounded exponent vectors having the same total
	on each connected component are connected by these moves.
	Indeed, all capacities $\alpha_w$ are positive on
	$U$, so units can be moved along paths of a spanning
	tree, successively matching the desired exponents
	at its leaves.
	
	The multihomogeneous component $J_\alpha$ is spanned
	by the differences arising from these moves.
	Consequently, a basis of $(S/J)_\alpha$ is indexed by
	the possible totals of the $x$-exponents on the
	components of $G_n[U]$. Thus
	\[
	\dim_{\K}(S/J)_\alpha
	=
	\prod_{C\in\pi(U)}
	\left(1+\sum_{w\in C}\alpha_w\right).
	\]
	
	For a component with $r$ vertices, summing over all
	positive column degrees gives
	\[
	\begin{aligned}
		\sum_{\alpha_1,\ldots,\alpha_r\geq1}
		(1+\alpha_1+\cdots+\alpha_r)
		t^{\alpha_1+\cdots+\alpha_r}&=
		\left(\frac{t}{1-t}\right)^r
		+r\frac{t}{(1-t)^2}
		\left(\frac{t}{1-t}\right)^{r-1}=
		\frac{t^r(r+1-t)}{(1-t)^{r+1}}.
	\end{aligned}
	\]
	Multiplying these expressions over the components
	of $G_n[U]$, and using
	$\sum_{C\in\pi(U)}|C|=|U|$, yields the summand
	corresponding to $U$. Summing over all supports
	$U\subseteq V(G_n)$ proves the formula.
\end{proof}
\begin{example}
	Let $n=5$. Applying Proposition~\ref{prop:Hilbert} to
	$G_5$, the bipartite complement of $C_{10}$, gives
	\[
	H_{S_{G_5}/J_{G_5}}(t)
	=
	\frac{
		1+9t+30t^2+30t^3-50t^4-118t^5
		+48t^6+100t^7-10t^8-30t^9
	}{(1-t)^{11}}.
	\]
	The numerator evaluates to $10$ at $t=1$, in agreement
	with $\dim(S_{G_5}/J_{G_5})=11$ and
	$e(S_{G_5}/J_{G_5})=10$.
\end{example}

\subsection{Projective dimension and depth}\label{sec:depth}
The big-height computation immediately gives a lower bound for the
projective dimension.

\begin{proposition}\label{prop:pd-depth-bounds}
For $n\geq5$, \(  4n-8\leq\pd(S_{G_n}/J_{G_n})\leq4n-6 \)
and hence,
\( 6\leq\depth(S_{G_n}/J_{G_n})\leq8. \)

\end{proposition}

\begin{proof}
For every homogeneous ideal, the projective dimension of its quotient
is at least the big height.  Corollary~\ref{cor:height-bight} therefore 
implies that \( \pd(S_{G_n}/J_{G_n})\geq4n-8. \) Since $S_{G_n}$ has dimension $4n$, the Auslander--Buchsbaum formula
yields \( \depth(S_{G_n}/J_{G_n})\leq8. \)
It remains to prove the lower bound for the depth.  By
\cite{SmallDepth}, a connected graph $H$ on at least four vertices
satisfies $\depth(S_H/J_H)=4$ precisely when
$H=H'\ast2K_1$ for some graph $H'$.  Such a graph contains two
nonadjacent vertices of degree $|V(H)|-2$.  Since every vertex of $G_n$
has degree $n-2<2n-2$, the graph $G_n$ is not of this form.  Hence
$\depth(S_{G_n}/J_{G_n})\neq4$.

The graphs whose binomial edge ideals have depth five are the
$D_5$-type graphs described in \cite{OnDepth}.  We verify that $G_n$
is not of any of the three possible types.  It cannot be $H\ast3K_1$, because each of the three independent
vertices would have degree $2n-3$. It cannot be
$H\ast(K_1\mathbin{\dot\cup}K_2)$ either, because the isolated vertex
of $K_1\mathbin{\dot\cup}K_2$ has degree $2n-3$ in the join, whereas
every vertex of $G_n$ has degree $n-2$.

For the remaining $D_5$-type, there would be two nonadjacent vertices
$a,b$ satisfying \( V(G_n)\setminus\{a,b\}=N_{G_n}(a)\cup N_{G_n}(b). \) However,
\[
\bigl|N_{G_n}(a)\cup N_{G_n}(b)\bigr|
\leq2(n-2)<2n-2,
\]
which is impossible.  Thus $G_n$ is not a $D_5$-type graph,
and therefore its depth is not five. The general lower bound
$\depth(S_H/J_H)\geq4$ for connected graphs on at least four vertices,
together with the
classifications just used, yields
\[
 \depth(S_{G_n}/J_{G_n})\geq6.
\]
The projective-dimension upper bound follows again from the
Auslander--Buchsbaum formula.
\end{proof}

The preceding proposition leaves three possibilities:
\[
 \pd(S_{G_n}/J_{G_n})\in\{4n-8,4n-7,4n-6\}.
\]
The crown-graph argument in \cite{Crown} rules out only depths four and
five, which is sufficient there because the crown graph has big height
$4n-6$.  For $G_n$, two additional possibilities must be excluded.
The preceding proposition leaves three possible values
for the depth. Thus
\[
\depth(S_{G_n}/J_{G_n})\in\{6,7,8\},
\qquad
\pd(S_{G_n}/J_{G_n})\in\{4n-8,4n-7,4n-6\}.
\]
To establish
\[
\pd(S_{G_n}/J_{G_n})=\bight(J_{G_n})=4n-8,
\]
it remains to prove
\[
H_{\mathfrak m}^{6}(S_{G_n}/J_{G_n})
=
H_{\mathfrak m}^{7}(S_{G_n}/J_{G_n})
=
0,
\]
where $\mathfrak m$ is the homogeneous maximal ideal.

%

\section{Connected domination and Vasconcelos numbers}\label{sec:vnumber}

A set $D\subseteq V(G)$ is a \emph{connected dominating set} if
$G[D]$ is connected and every vertex outside $D$ is adjacent to a
vertex in $D$.  The minimum cardinality of such a set is the connected
domination number $\gamma_c(G)$. We first compute the connected domination number of $G_n$.

\begin{proposition}\label{prop:connected-domination}
	The connected domination number of $G_n$ is
	$\gamma_c(G_n)=4$.
\end{proposition}

\begin{proof}
	We first show that no set of at most three vertices
	is a connected dominating set.
	A singleton does not dominate $G_n$, since it has no
	neighbours in its own bipartition class.
	
	Now let $D\subseteq V(G_n)$ satisfy $2\leq |D|\leq3$,
	and suppose that $G_n[D]$ is connected.
	Then $D$ meets both bipartition classes, and one of
	these classes contains exactly one vertex of $D$.
	By symmetry, we may assume that $D\cap Y=\{v_j\}$.
	Connectivity forces every vertex of $D\cap X$ to be
	adjacent to $v_j$ in $G_n$.
	Consequently, the two vertices of
	$\N_{F_n}(v_j)=\{u_j,u_{j+1}\}$ lie outside $D$.
	Neither has a neighbour in $D$, because neither is
	adjacent to $v_j$ and all other vertices of $D$
	belong to $X$.
	Thus $D$ is not a dominating set, proving that
	$\gamma_c(G_n)\geq4$.
	
	For the reverse inequality, consider \( D=\{u_0,u_2,v_0,v_3\}. \)
	The induced subgraph $G_n[D]$ is the path
	$u_0-v_3-u_2-v_0$, so it is connected.
	Moreover, \( \N_{F_n}(v_0)=\{u_0,u_1\}, \)
     and \( \N_{F_n}(v_3)=\{u_3,u_4\}. \)
	These sets are disjoint, so every vertex of $X$
	is adjacent in $G_n$ to at least one of $v_0,v_3$. Similarly, \( \N_{F_n}(u_0)=\{v_0,v_{n-1}\} \) and \( \N_{F_n}(u_2)=\{v_1,v_2\} \)
	are disjoint, and hence every vertex of $Y$ is
	adjacent in $G_n$ to at least one of $u_0,u_2$.
	Therefore $D$ is a connected dominating set, giving
	$\gamma_c(G_n)\leq4$.
	The two inequalities establish the assertion.
\end{proof}

\begin{definition}
For $P\in\Min(J_G)$, the local $v$-number of $J_G$ at $P$ is
\[
v_P(J_G)
=
\min\{\deg f:f\text{ is homogeneous and }J_G:f=P\}.
\]
We call such an $f$ a separator for $P$. The $v$-number of $J_G$ is
\[
v(J_G)=\min_{P\in\Min(J_G)}v_P(J_G).
\]
Since $J_G$ is radical, $\Ass(S_G/J_G)=\Min(J_G)$.
\end{definition}
It was proved in \cite{ConnectedDomination} that for every connected
noncomplete graph $G$, \(   v_{P_\varnothing(G)}(J_G)=\gamma_c(G). \) Hence, Proposition~\ref{prop:connected-domination} gives the
following.

\begin{corollary}\label{cor:empty-local-v}
For $n\geq5$, \(  v_{P_\varnothing(G_n)}(J_{G_n})=4. \)
In particular, $ v(J_{G_n})\leq4$.
\end{corollary}

\subsection{Separators, support, and component merging}
Write $\br{a}{b}=x_a y_b-x_b y_a$.
By the radical minimal-prime decomposition, a polynomial $f$ is a
separator for $P_T(G)$ if and only if \( f\notin P_T(G) \) and 
\[
f\in P_Q(G)
\quad\text{for every }
Q\in\cC(G)\setminus\{T\}.
\]
The radical minimal-prime decomposition shows that this is equivalent
to $J_G:f=P_T$. The least degree of a separator is the local $v$-number.

\begin{lemma}\label{lem:multiseparator}
	Let $T\in\cC(G)$ and put $U=V(G)\setminus T$.
	A separator for $P_T(G)$ of minimum total degree
	can be chosen multihomogeneous with respect to
	the column grading $\deg x_w=\deg y_w=e_w$.
	If $\alpha$ is its multidegree, then
	$\operatorname{supp}(\alpha)\subseteq U$.
	Moreover, if $A\subseteq V(G)$ satisfies
	$P_T(G)\not\subseteq P_A(G)$, then every separator
	for $P_T(G)$ belongs to $P_A(G)$.
\end{lemma}

\begin{proof}
	Write $J=J_G$ and $P_B=P_B(G)$.
	Since $J$ is radical, its minimal-prime decomposition
	shows that a polynomial $f$ is a separator for $P_T$
	if and only if \( f\notin P_T \) and \( f\in P_B \) for every \( B\in\cC(G)\setminus\{T\}. \)
	Indeed, this follows by taking colons in the minimal-prime decomposition of $J$.
	Choose a separator $f$ of minimum total degree and write $f=\sum_\alpha f_\alpha$ as a sum of its column-homogeneous components.
	Every minimal prime of $J$ is column-homogeneous.
	Hence each $f_\alpha$ belongs to every minimal
	prime other than $P_T$.
	Since $f\notin P_T$, at least one component
	$f_\alpha$ does not belong to $P_T$.
	This component is therefore a separator.
	Its total degree is at most that of $f$, so the
	minimality of $\deg f$ implies that $f_\alpha$
	also has minimum degree.
	If $\alpha_w>0$ for some $w\in T$, then every
	monomial of $f_\alpha$ is divisible by $x_w$ or
	$y_w$. Consequently,
	$f_\alpha\in(x_w,y_w)\subseteq P_T$, a contradiction.
	Thus $\operatorname{supp}(\alpha)\subseteq U$.
	
	Finally, let $A\subseteq V(G)$ satisfy
	$P_T\not\subseteq P_A$.
	Since $P_A$ is a prime containing $J$, it contains
	a minimal prime $P_B$ of $J$.
	The hypothesis implies that $P_B\neq P_T$.
	Every separator for $P_T$ belongs to $P_B$,
	and therefore belongs to $P_A$.
\end{proof}

The following formula is due to
\cite[Theorem~4.15]{Liwski}. We include a direct
proof using supports of multihomogeneous separators.

\begin{lemma}\label{lem:twocomponents}
	Let $G$ be connected, and let $T\in\cC(G)$ such
	that $G\setminus T$ has exactly two connected components,
	with vertex sets $C_1$ and $C_2$.
	For $i=1,2$, let $d_i$ be the minimum cardinality
	of a connected dominating set of $G[C_i\cup T]$
	contained in $C_i$. Then \( {v}_{P_T(G)}(J_G)=d_1+d_2. \)
\end{lemma}

\begin{proof}
	Write $P_A=P_A(G)$.
	Since $G$ is connected, $T\neq\varnothing$.
	The cut point property implies that every vertex
	of $T$ has a neighbour in each $C_i$.
	Consequently, $C_i$ itself is a connected
	dominating set of $G[C_i\cup T]$, so $d_i$
	is well defined.
	
	By Lemma~\ref{lem:multiseparator}, choose a
	minimum-degree column-homogeneous separator $f$
	for $P_T$, with support
	$W\subseteq V(G)\setminus T$.
	Put $W_i=W\cap C_i$ and
	$S_W=\mathbb K[x_w,y_w:w\in W]$.
	We show that $W_i$ is connected and dominates
	$G[C_i\cup T]$ for each $i$.
	First, every $t\in T$ has a neighbour in each $W_i$.
	Otherwise, put $A=V(G)\setminus(W\cup\{t\})$.
	Since $x_t\in P_T\setminus P_A$,
	Lemma~\ref{lem:multiseparator} gives $f\in P_A$.
	However, $t$ has neighbours in at most one of
	$W_1,W_2$. Thus no connected component of
	$G[W\cup\{t\}]$ meets both $W_1$ and $W_2$.
	It follows that $P_A\cap S_W\subseteq P_T$,
	contradicting $f\notin P_T$.
	Since $T\neq\varnothing$, both $W_1$ and $W_2$
	are nonempty.
	Next, suppose that $G[W_i]$ is disconnected.
	Choose $a,b\in W_i$ in different connected
	components and put $A=V(G)\setminus W$.
	Then $\br{a}{b}\in P_T\setminus P_A$.
	Moreover, every component of $G[W]$ is contained
	in one of $C_1,C_2$, so
	$P_A\cap S_W\subseteq P_T$.
	Lemma~\ref{lem:multiseparator} again gives the
	contradiction $f\in P_T$.
	Hence each $G[W_i]$ is connected.
	Finally, suppose that some $w\in C_i\setminus W_i$
	has no neighbour in $W_i$.
	Set $A=V(G)\setminus(W\cup\{w\})$.
	The vertex $w$ is isolated in $G[W\cup\{w\}]$.
	Choosing $a\in W_i$, we obtain
	$\br{a}{w}\in P_T\setminus P_A$.
	As before, $P_A\cap S_W\subseteq P_T$,
	which contradicts Lemma~\ref{lem:multiseparator}.
	Thus $W_i$ dominates $C_i$, and the first part
	of the argument shows that it also dominates $T$.
	Therefore \( \deg f\geq |W|=|W_1|+|W_2|\geq d_1+d_2. \)
	
	Conversely, choose sets $D_i\subseteq C_i$
	attaining $d_i$, and put $D=D_1\cup D_2$.
	Choose $a\in D_1$ and $b\in D_2$, and define \( f=\br{a}{b}
	\prod_{w\in D\setminus\{a,b\}}x_w. \)
	Since $a,b$ belong to different components of
	$G\setminus T$, the bracket $\br{a}{b}$ does not belong
	to $P_T$. None of the other factors belongs
	to $P_T$ either. As $P_T$ is prime, $f\notin P_T$.
	
	Let $Q\in\cC(G)\setminus\{T\}$.
	If $Q\cap D\neq\varnothing$, then $f\in P_Q$.
	Suppose that $Q\cap D=\varnothing$.
	Each $D_i$ is connected in $G\setminus Q$.
	If $D_1$ and $D_2$ lie in the same component,
	then $\br{a}{b}\in P_Q$, and hence $f\in P_Q$. Otherwise, every vertex of $T$ belongs to $Q$:
	a surviving vertex of $T$ would have neighbours
	in both $D_1$ and $D_2$, joining their components.
	Since $D_i$ dominates $C_i$, every surviving
	vertex of $C_i$ belongs to the component of $D_i$.
	If $q\in Q\setminus T$, then $q$ belongs to some
	$C_i$ and has no neighbours in the other component.
	Restoring $q$ therefore cannot make it a cut vertex,
	contrary to $Q\in\cC(G)$.
	Thus $Q=T$, a contradiction.
	
	Therefore, $f$ belongs to every minimal prime
	other than $P_T$ and lies outside $P_T$.
	It is therefore a separator of degree
	$|D|=d_1+d_2$, proving the reverse inequality.
\end{proof}
\begin{lemma}\label{lem:mixed-degree}
	Let $G$ be bipartite with bipartition $X\sqcup Y$, and let
	$T\in\cC(G)$ meet both classes. If $f$ is a column-homogeneous
	separator for $P_T(G)$ with multidegree $\alpha$, then
	\( 	|\operatorname{supp}(\alpha)\cap X|\geq2
	~\text{and}~
	|\operatorname{supp}(\alpha)\cap Y|\geq2. \)
	In particular,
	\[
	\sum_{w\in X}\alpha_w\geq2
	~\text{and}~
	\sum_{w\in Y}\alpha_w\geq2,
	\]
	and therefore $v{P_T(G)}(J_G)\geq4$.
\end{lemma}

\begin{proof}
	Write $J=J_G$ and choose $t\in T\cap X$.
	The cut point property implies that $t$ is
	an internal vertex of $G$.
Let $G_t$ be obtained from $G$ by completing the neighbourhood of
$t$. By Ohtani's decomposition~\cite{Ohtani},
\[
J
=
J_{G_t}
\cap
\bigl((x_t,y_t)+J_{G\setminus t}\bigr).
\]
The vertex $t$ is simplicial in $G_t$, and hence it belongs to no
cut point set of $G_t$. Therefore $x_t$ is contained in no minimal
prime of $J_{G_t}$ and is a nonzerodivisor modulo $J_{G_t}$.
Consequently,
\[
J:x_t=J_{G_t}.
\] Since $(J:f)=P_T(G)$ and $x_t\in P_T(G)$,
	we have $x_tf\in J$, and hence $f\in J_{G_t}$.
	The ideal $J_{G_t}$ is obtained from $J$ by
	adjoining minors $\br{a}{b}$ between distinct
	vertices $a,b\in N_G(t)\subseteq Y$.
	Every such minor has column multidegree $e_a+e_b$. If $\operatorname{supp}(\alpha)$ contained at most
	one vertex of $Y$, none of these additional
	generators could contribute to the
	multidegree-$\alpha$ component of $J_{G_t}$.
	It would follow that $f\in J$, contradicting
	the fact that $f$ is a separator.
	Thus $|\operatorname{supp}(\alpha)\cap Y|\geq2$. Choosing a vertex of $T\cap Y$ and interchanging
	the roles of $X$ and $Y$ gives
	$|\operatorname{supp}(\alpha)\cap X|\geq2$.
	The degree inequalities follow immediately.
	Finally, Lemma~\ref{lem:multiseparator} allows a
	minimum-degree separator to be chosen
	column-homogeneous, so
	$v_{P_T(G)}(J_G)\geq4$.
\end{proof}

Let $T\in\mathcal \cC(G)$, and let $f$ be a column-homogeneous
separator for $P_T(G)$ of multidegree $\alpha$, with support $W$.
Suppose that the connected components of $G\setminus T$ induce the partition
$W=A\sqcup B\sqcup C$ into three nonempty blocks. Assume that there exist
$Q_{AB},Q_{AC},Q_{BC}\in\mathcal C(G)$, each disjoint from $W$,
whose corresponding surviving graphs induce the partitions
\[
\{A\cup B,C\},~
\{A\cup C,B\},~
\{B\cup C,A\},
\]
respectively.
Thus each of these cut point sets merges exactly one pair of the
original blocks.
For $D\in\{A,B,C\}$, put
$d_D=\sum_{w\in D}\alpha_w$.

\begin{lemma}\label{lem:mergers}
	With the above notation and assumptions,
	$d_A,d_B,d_C\geq2$. In particular, $\deg f\geq6$.
\end{lemma}

\begin{proof}
	Write $S_W=\mathbb K[x_w,y_w:w\in W]$.
	For each block $D\in\{A,B,C\}$, introduce
	variables $z_{D1},z_{D2}$, and consider the
	homomorphism
	\[
	\varphi:S_W\longrightarrow
	\mathbb K[t_w:w\in W]
	[z_{A1},z_{A2},z_{B1},z_{B2},z_{C1},z_{C2}]
	\]
	defined, for $w\in D$,  by \( 	\varphi(x_w)=t_wz_{D1},
	~\text{and}~
	\varphi(y_w)=t_wz_{D2} \). The standard rank-one parametrization gives
	$\ker\varphi=P_T(G)\cap S_W$.
	Since $f\notin P_T(G)$ and $f$ has column
	multidegree $\alpha$, we may write
	\[
	\varphi(f)=t^\alpha F(z),
	~
	t^\alpha=\prod_{w\in W}t_w^{\alpha_w},
	\]
	where $F\neq0$.
	Moreover, $F$ is homogeneous of degree
	$d_D=\sum_{w\in D}\alpha_w$ in the pair
	$(z_{D1},z_{D2})$ for each block $D$.
	For distinct blocks $D,E$, put \( \Delta_{DE}=z_{D1}z_{E2}-z_{D2}z_{E1}. \)
	The partition associated with $Q_{AB}$ differs
	from that associated with $T$, so $Q_{AB}\neq T$.
	Since $f$ is a separator for $P_T(G)$,
	we have $f\in P_{Q_{AB}}(G)\cap S_W$.
	This intersection is generated by the minors
	within $A\cup B$ and those within $C$.
	Under $\varphi$, minors within an individual
	block vanish, while a minor between
	$a\in A$ and $b\in B$ maps to
	$t_at_b\Delta_{AB}$.
	Consequently, $\Delta_{AB}$ divides
	$t^\alpha F(z)$, and hence divides $F(z)$.
	
	Applying the same argument to $Q_{AC}$ and
	$Q_{BC}$ shows that $\Delta_{AC}$ and
	$\Delta_{BC}$ also divide $F$.
	These three determinants are pairwise
	nonassociate irreducible polynomials.
	Since the polynomial ring is a unique
	factorization domain, we see that \( \Delta_{AB}\Delta_{AC}\Delta_{BC}\mid F. \)
	This product has degree two in each of the
	three pairs of variables. Therefore
	$d_A,d_B,d_C\geq2$.
\end{proof}

\subsection{Local Vasconcelos numbers}

Throughout this subsection, write $S=S_{G_n}$,
$J=J_{G_n}$, and $P_T=P_T(G_n)$.
The types of cut point sets refer to
Theorem~\ref{thm:cut-point-classification}.
For $W\subseteq V(G_n)$, put
$S_W=\K[x_w,y_w:w\in W]$.

By Lemma~\ref{lem:multiseparator}, a separator of minimum
degree may be chosen column-homogeneous.
If $W$ is its support, then
$W\subseteq V(G_n)\setminus T$ and $\deg f\geq|W|$.
We shall repeatedly use the following observation:
if $Q\in\cC(G_n)\setminus\{T\}$ avoids $W$ and the
components of $G_n\setminus Q$ and $G_n\setminus T$
induce the same partition of $W$, then
\[
P_Q\cap S_W=P_T\cap S_W.
\]
This contradicts $f\in P_Q\setminus P_T$.
We refer to this as the support observation.

In checking the explicit separators below, a prime
deleting a supported column contains the polynomial
by column homogeneity.
Moreover, every displayed product containing a bracket
belongs to $P_{\varnothing}$.

\begin{lemma}\label{lem:local-v-two-components}
	If $T$ is of type \textup{(iv)}, then
	$ v_{P_T}(J)=5$.
	If $T$ is of type \textup{(vi)} or \textup{(vii)}, then
	$ v_{P_T}(J)=6$.
\end{lemma}

\begin{proof}
	In each case, $G_n\setminus T$ has exactly two components.
	Let $C_1,C_2$ denote their vertex sets.
	By Lemma~\ref{lem:twocomponents},
	$ v_{P_T}(J)=d_1+d_2$, where $d_i$ is the
	minimum size of a connected subset of $C_i$ dominating
	$G_n[C_i\cup T]$.
	The component criterion shows that each full set $C_i$
	is an admissible choice.
	
	\smallskip
	\noindent\emph{Type \textup{(iv)}.}
	By symmetry, suppose
	$U=X\cup\{v_0,v_1\}$.
	The components have vertex sets $\{u_1\}$ and
	$C=U\setminus\{u_1\}$.
	Since $N_{F_n}(u_1)=\{v_0,v_1\}$, the singleton
	$\{u_1\}$ dominates $T$, and its minimum is one.
	
	Consider a connected set $D\subseteq C$ dominating
	$G_n[C\cup T]$.
	It must contain a $Y$-vertex: otherwise connectivity
	would force $D$ to be a singleton in $X$, which cannot
	dominate the other vertices of $C\cap X$.
	In fact, $D$ must contain both $v_0$ and $v_1$.
	If $v_0$ were its only $Y$-vertex, then $u_0$ would
	have to belong to $D$, but would be isolated in $G_n[D]$.
	The case in which $v_1$ is the only $Y$-vertex is
	similar, using $u_2$. The set $D$ must also contain at least two $X$-vertices.
	Indeed, if $w$ were its only $X$-vertex, domination of
	$T=Y\setminus\{v_0,v_1\}$ would require
	$N_{F_n}(w)\subseteq\{v_0,v_1\}$.
	This forces $w=u_1$, contrary to $D\subseteq C$.
	Thus $|D|\geq4$.
	
	Conversely, the set $\{u_0,u_3,v_0,v_1\}$ is connected
	and dominates $G_n[C\cup T]$.
	The vertices $v_0,v_1$ dominate $X\setminus\{u_1\}$. Also, \( N_{F_n}(u_0)=\{v_0,v_{n-1}\} \) and \( N_{F_n}(u_3)=\{v_2,v_3\} \)
	are disjoint.
	The two minima are therefore one and four, giving
	$ v_{P_T}(J)=5$.
	
	\smallskip
	\noindent\emph{Type \textup{(vi)}.}
	The two components are three-vertex stars.
	Since $T$ meets both bipartition classes, every
	admissible dominating set in either component must
	meet both classes.
	It must therefore contain the centre and at least
	one leaf.
	Each leaf has one $F_n$-neighbour in $T$.
	If only one leaf were chosen, this neighbour would
	not be dominated: the centre lies in the same
	bipartition class as that neighbour.
	Hence both leaves are necessary.
	Both minima equal three, so
	$ v_{P_T}(J)=6$.
	
	\smallskip
	\noindent\emph{Type \textup{(vii)}.}
	The components are $K_2$ and either $P_4$ or $K_{2,2}$.
	An admissible set in the $K_2$ component must meet
	both bipartition classes, so its minimum is two. Suppose the other component is $K_{2,2}$.
	Each of its vertices has exactly one $F_n$-neighbour
	in $U$ and hence another in $T$.
	Choosing only one vertex from either bipartition
	class would leave that vertex's $F_n$-neighbour in
	$T$ undominated.
	Thus both vertices in each class are necessary.
	
	Suppose instead that the component is $P_4$.
	For either leaf, the other vertex in its bipartition
	class has exactly one $F_n$-neighbour in $U$ and
	another in $T$.
	Omitting the leaf would therefore leave a vertex
	of $T$ undominated.
	Both leaves are necessary, and connectivity then
	forces the two internal vertices.
	The minimum for the four-vertex component is four
	in either case.
	Hence $ v_{P_T}(J)=2+4=6$.
\end{proof}

\begin{lemma}\label{lem:local-v-one-survivor}
	If $T$ is of type \textup{(iii)}, then
	$ v_{P_T}(J)=5$.
\end{lemma}

\begin{proof}
	By symmetry, suppose $U=X\cup\{v_0\}$, and write
	$a=u_0$, $b=u_1$, $c=u_2$, and $d=u_4$.
	The components of $G_n[U]$ are the singletons
	$\{a\}$ and $\{b\}$ and the star containing $v_0$. We first prove the upper bound.
	Consider \( f=x_{v_0}\br{a}{c}\br{b}{d}. \)
	Neither bracket belongs to $P_T$, and $x_{v_0}\notin P_T$.
	Since $P_T$ is prime, $f\notin P_T$.
	
	A prime deleting a supported column contains $f$,
	and $f\in P_{\varnothing}$.
	Consider a nonempty one-sided cut-point set $Q$
	whose survivor set contains the support.
	Apart from $Q=T$, the only possibilities are \( X\cup\{v_{n-1},v_0\} \) or \( X\cup\{v_0,v_1\}. \) In either case, one bracket has both endpoints in
	the nontrivial component, so $f\in P_Q$.
	A mixed survivor set containing the support would
	have to be of type \textup{(vi)}, since it contains
	four $X$-vertices.
	Its other $Y$-centre would have to have both $u_2$
	and $u_4$ as $F_n$-neighbours.
	This is impossible because these vertices are not
	cyclically consecutive.
	Thus $f$ is a separator, and
	$ v_{P_T}(J)\leq5$.
	
	For the reverse inequality, suppose that a
	column-homogeneous separator has degree at most four,
	and let $W$ be its support.
	Since it belongs to the distinct minimal prime
	$P_Y=(x_y,y_y:y\in Y)$, its support contains $v_0$.
	The support must also contain both $a$ and $b$.
	If either were omitted, one of the two type
	\textup{(iv)} survivor sets displayed above would
	induce the same partition on $W$ as the target
	prime, contradicting the support observation.
	If $W=\{a,b,v_0\}$, the type \textup{(v)} survivor set \( \{u_0,u_1,u_2,v_0,v_1\} \)
	likewise induces the target partition on $W$.
	
	Consequently, \( W=\{a,b,c',v_0\},
	~ c'\in X\setminus\{a,b\}, \)
	and the multidegree is squarefree.
	If $c'=u_2$, the preceding type \textup{(v)}
	configuration again gives a contradiction.
	If $c'=u_{n-1}$, use instead \( \{u_{n-1},u_0,u_1,v_{n-1},v_0\}. \)
	
	In the remaining cases, choose $w\in N_{F_n}(c')$.
	Then $w$ is neither $v_0$ nor a cyclic neighbour
	of $v_0$.
	The type \textup{(vi)} survivor set
	\( N_{F_n}(v_0)\cup N_{F_n}(w)\cup\{v_0,w\} \)
	induces on $W$ the partition
	$\{a,b\}\mid\{c',v_0\}$.
	Together with the two type \textup{(iv)} survivor
	sets above, this realizes all three pairwise mergers
	of the target partition \( \{a\}\mid\{b\}\mid\{c',v_0\}. \)
	Lemma~\ref{lem:mergers} forces degree at least two
	on each singleton block $\{a\}$ and $\{b\}$,
	contradicting squarefreeness.
	Therefore $ v_{P_T}(J)=5$.
\end{proof}

\begin{lemma}\label{lem:local-v-pfive}
	If $T$ is of type \textup{(v)}, then
	$ v_{P_T}(J)=6$.
\end{lemma}
\begin{proof}
	By symmetry, write the deleted path as
	$a-d-b-e-c$, where $a,b,c\in X$ and $d,e\in Y$.
	The components of $G_n[U]$ have vertex sets \( A=\{a,e\}, ~ B=\{c,d\},~\text{and}~ C=\{b\}. \) Consider \( f=\br{a}{b}\br{b}{c}\br{d}{e}. \)
	It has degree six and support $U$.
	Each bracket joins two distinct components, so
	none belongs to $P_T$.
	The primality of $P_T$ gives $f\notin P_T$.
	
	Any other minimal prime deleting a supported
	column contains $f$.
	If a different minimal prime avoids $U$, its survivor
	set strictly contains $U$, and hence contains some
	$t\in T$.
	Restoring $t$ joins at least two of $A,B,C$ by the
	component criterion.
	The product $f$ has a bracket joining each pair
	of these components, so it belongs to that prime.
	Thus $f$ is a separator and
	$ v_{P_T}(J)\leq6$.
	
	Suppose now that a column-homogeneous separator has
	degree at most five, and let $W$ be its support.
	Lemma~\ref{lem:mixed-degree} implies that $d,e\in W$
	and that at least two of $a,b,c$ belong to $W$.
	In fact, all three must belong to $W$.
	To see this, suppose one of $a,b,c$ is omitted,
	and let $r,s$ be the other two.
	The survivor set \( \{r,s\}\cup N_{F_n}(r)\cup N_{F_n}(s) \)
	is of type \textup{(v)} or \textup{(vi)}.
	It contains $W$ and induces the same partition
	on $W$ as the target prime.
	Indeed, the pairs $(r,s)$ are $(b,c)$, $(a,c)$,
	and $(a,b)$ when $a$, $b$, and $c$ are omitted,
	respectively.
	The support observation gives a contradiction.
	
	It follows that $W=U$ and the separator has
	squarefree multidegree.
	The type \textup{(iv)} survivor set
	$X\cup\{d,e\}$ merges $A$ with $B$.
	Extending the deleted path by one vertex at either
	end gives two type \textup{(vii)} survivor sets,
	which merge $C$ with $A$ and $C$ with $B$,
	respectively.
	Lemma~\ref{lem:mergers} therefore forces degree at
	least two on the singleton block $\{b\}$,
	contradicting squarefreeness.
	Hence $ v_{P_T}(J)=6$.
\end{proof}

\begin{lemma}\label{lem:local-v-part}
	Let $T=X$ or $T=Y$.
	Then $ v_{P_T}(J)\geq4$.
	If $n\geq6$, equality holds.
\end{lemma}

\begin{proof}
	By symmetry, take $T=Y$.
	A column-homogeneous separator is supported on $X$.
	For each $i\in\mathbb Z_n$, the type \textup{(iii)}
	prime whose survivor set is $X\cup\{v_i\}$ restricts
	to the ideal \( I_i=J_{K_{X\setminus\{u_i,u_{i+1}\}}} \)
	in $S_X$. Thus every separator belongs to each $I_i$. Suppose a separator has degree at most three,
	and let $W$ be its support.
	If $|W|\leq2$, choose a cyclic pair
	$\{u_i,u_{i+1}\}$ leaving at most one vertex of $W$
	outside the pair.
	Then $I_i\cap S_W=0$, a contradiction.
	
	Suppose $|W|=3$.
	If $W$ contains two cyclically consecutive vertices,
	the same argument applies.
	Otherwise, write $W=\{a,b,c\}$.
	For each of these three vertices, choose a cyclic
	pair containing it.
	Such a pair contains neither of the other two vertices.
	Membership in the corresponding ideals therefore
	forces divisibility by \( \br{a}{b},~ \br{a}{c},~\text{and}~ \br{b}{c}. \)
	These are pairwise nonassociate irreducible polynomials,
	so their product divides the separator.
	Its degree is consequently at least six, again
	a contradiction.
	This proves $ v_{P_Y}(J)\geq4$.
	
	Now suppose $n\geq6$, and consider \( f=\br{u_0}{u_1}\br{u_3}{u_4}. \)
	This is a nonzero polynomial in $S_X$, so
	$f\notin P_Y$.
	Among the four indices $0,1,3,4$, the only cyclically
	consecutive pairs are $\{0,1\}$ and $\{3,4\}$.
	Hence any cyclic pair $\{i,i+1\}$ meets at most
	one of these two pairs.
	At least one bracket factor lies in $I_i$,
	which verifies membership in every type
	\textup{(iii)} prime.
	
	A type \textup{(iv)} prime avoiding the support
	has only one isolated $X$-vertex.
	At least one bracket therefore lies entirely in
	its nontrivial component.
	A mixed survivor set containing the support must
	be of type \textup{(vi)}.
	Its two deleted-cycle neighbourhood pairs are
	$\{u_0,u_1\}$ and $\{u_3,u_4\}$.
	Each pair lies in a component of the corresponding
	graph $G_n[U]$, so $f$ belongs to this prime as well.
	The empty-set prime contains $f$, and every remaining
	minimal prime other than $P_Y$ deletes a supported
	column.
	Thus $f$ is a separator of degree four, proving
	$ v_{P_Y}(J)=4$.
\end{proof}

\begin{lemma}\label{lem:local-v-part-five}
	If $n=5$ and $T=X$ or $T=Y$, then
	$ v_{P_T}(J)=6$.
\end{lemma}

\begin{proof}
	By symmetry, take $T=Y$.
	The polynomial \( f=\br{u_0}{u_1}\br{u_2}{u_3}\br{u_4}{u_0} \)
	is nonzero and supported on all five $X$-vertices,
	so $f\notin P_Y$.
	Every cyclic pair misses at least one bracket pair,
	which verifies membership in every type
	\textup{(iii)} prime.
	For a type \textup{(iv)} prime, at least one bracket
	avoids its isolated $X$-vertex and therefore belongs
	to its nontrivial component.
	No mixed survivor set contains all five $X$-vertices.
	Together with membership in $P_{\varnothing}$,
	these observations show that $f$ is a separator.
	Hence $ v_{P_Y}(J)\leq6$.
	
	For the reverse inequality, suppose a
	column-homogeneous separator has degree at most five,
	and write $W$ for its support.
	As in Lemma~\ref{lem:local-v-part}, it belongs to
	\[
	I_i=J_{K_{X\setminus\{u_i,u_{i+1}\}}},
	~ i\in\mathbb Z_5.
	\]
	If $|W|\leq2$, some $I_i$ has zero intersection
	with $S_W$.
	The same holds when $|W|=3$, since every three
	vertices in the cyclic order on $X$ contain a
	consecutive pair.
	Thus $|W|\geq4$.
	
	If $|W|=4$, the consecutive pairs contained in $W$
	form a path on four vertices.
	Each of its three edges leaves a complementary pair
	in $W$, and membership in the corresponding $I_i$
	forces the bracket on that pair to divide the separator.
	The three brackets are distinct irreducible polynomials.
	Their product has degree six, a contradiction.
	
	It remains to consider $W=X$.
	The separator must then have degree five and squarefree
	column multidegree.
	All the ideals $I_i$ are homogeneous for the bigrading
	$\deg x_w=(1,0)$ and $\deg y_w=(0,1)$.
	We may therefore examine each total $x$-degree $r$
	separately.
	
	For $r=0$ or $r=5$, the only monomial does not belong
	to any $I_i$.
	For $r=1$, membership in an $I_i$ isolating a specified
	column forces the coefficient of the monomial whose
	$x$ occurs in that column to vanish.
	Thus all coefficients vanish.
	The case $r=4$ follows by exchanging $x$ and $y$.
	
	For $r=2$, write \( f=\sum_{0\leq i<j\leq4}
	c_{ij}x_{u_i}x_{u_j}
	\prod_{k\notin\{i,j\}}y_{u_k}, \)
	where $c_{ji}=c_{ij}$ for distinct indices.
	Membership in $I_i$ forces $c_{i,i+1}=0$,
	with indices read modulo five.
	It also forces the sum of the three coefficients
	on the complementary triple
	$X\setminus\{u_i,u_{i+1}\}$ to be zero.
	Two of those pairs are cyclically consecutive,
	so their coefficients already vanish.
	The remaining coefficient therefore vanishes as well.
	As $i$ varies, all five nonconsecutive pairs occur
	in this way.
	Hence every coefficient is zero.
	The case $r=3$ follows by exchanging $x$ and $y$. No nonzero separator of degree at most five exists.
	Therefore $ v_{P_Y}(J)=6$.
\end{proof}
\begin{theorem}\label{thm:local-v}
	Let $T\in\cC(G_n)$, with types as in
	Theorem~\ref{thm:cut-point-classification}.
	Then the following hold:
	\begin{enumerate}[label=\textup{(\roman*)}]
		\item If $T$ is of type \textup{(i)}, then
		$ v_{P_T}(J)=4$.
		
		\item If $T$ is of type \textup{(ii)}, then
		$ v_{P_T}(J)=6$ when $n=5$, and
		$ v_{P_T}(J)=4$ when $n\geq6$.
		
		\item If $T$ is of type \textup{(iii)} or
		\textup{(iv)}, then $ v_{P_T}(J)=5$.
		
		\item If $T$ is of type \textup{(v)},
		\textup{(vi)}, or \textup{(vii)}, then
		$ v_{P_T}(J)=6$.
	\end{enumerate}
	Consequently, $ v(J)=4$.
\end{theorem}

\begin{proof}
	The value for type \textup{(i)} follows from
	Corollary~\ref{cor:empty-local-v}.
	The values for type \textup{(ii)} follow from
	Lemmas~\ref{lem:local-v-part}
	and~\ref{lem:local-v-part-five}.
	Type \textup{(iii)} is covered by
	Lemma~\ref{lem:local-v-one-survivor}, and type
	\textup{(v)} by Lemma~\ref{lem:local-v-pfive}.
	The remaining types follow from
	Lemma~\ref{lem:local-v-two-components}.
	The minimum of these local values is four.
\end{proof}

\section{Induced paths and regularity}\label{sec:regularity}

We determine the length of a longest induced path in $G_n$ and
deduce  bounds for $\reg(S_{G_n}/J_{G_n})$.  For a graph $G$, let $\ell(G)$ denote the length of a longest induced
path in $G$.

\begin{proposition}\label{prop:longest-induced-path}
	For every $n\geq5$, the length of a longest induced path in $G_n$
	is $6$.
\end{proposition}

\begin{proof}
	Consider the sequence \(  	u_0,\ v_1,\ u_{n-1},\ v_0,\ u_2,\ v_{n-1},\ u_1. \)
	Consecutive vertices are adjacent in $G_n$ by
	\eqref{eq:adjacency}. Moreover, every nonconsecutive pair belonging
	to opposite bipartition classes is an edge of the deleted cycle
	$F_n$. Hence the displayed sequence is an induced path of length
	$6$, and therefore \( \ell(G_n)\geq6. \)
	
	Suppose that $G_n$ contains an induced path of length at least $7$.
	Such a path contains at least four vertices from each bipartition
	class. Consider an endpoint $u$ of the path. Among the at least four
	vertices of the opposite bipartition class occurring on the path,
	the endpoint $u$ is adjacent to the first such vertex and, by
	inducedness, is nonadjacent to all the other three. This is
	impossible, because every vertex of $G_n$ has exactly two
	non-neighbours in the opposite bipartition class, namely its two
	neighbours in $F_n$. Hence no induced path has length at least $7$,
	and therefore $\ell(G_n)=6$.
\end{proof}

\begin{corollary}\label{cor:regularity}
	For every $n\geq5$, \( 6\leq \reg(S_{G_n}/J_{G_n})\leq2n-2. \)
\end{corollary}

\begin{proof}
	By Proposition~\ref{prop:longest-induced-path} and
	\cite[Corollary~2.3]{MatsudaMurai},
	\[
	\reg(J_{G_n})\geq 6+1=7.
	\]
	Since $J_{G_n}\neq0$, \( 	\reg(S_{G_n}/J_{G_n})=\reg(J_{G_n})-1\geq6. \)
	On the other hand, $G_n$ is not a path and has $2n$ vertices.
	By \cite[Theorem~3.2]{RegularityUpper},
	\[
	\reg(J_{G_n})\leq2n-1,
	\]
	and hence \( \reg(S_{G_n}/J_{G_n})\leq2n-2. \)
\end{proof}

\end{document}